\documentclass[12pt,authoryear,letterpaper,reqno,oneside]{amsart}
\pdfoutput=1
\usepackage[left=1.4in,right=1.4in,bottom=1.4in,top=1.40in]{geometry}
\usepackage{amsmath,amsfonts,amssymb}%
\usepackage[usenames]{color}
\usepackage{setspace}
\usepackage{comment}
\usepackage[colorlinks,citecolor=darkblue,urlcolor=darkblue,linkcolor=darkgreen]{hyperref}
\usepackage{stmaryrd}

\usepackage[capitalize]{cleveref}
	\crefformat{subsection}{\S#2#1#3}

\usepackage{microtype}

\crefname{lemma}{Lemma}{Lemmas}
\crefname{corollary}{Corollary}{Corollaries}
\crefname{theorem}{Theorem}{Theorems}
\crefname{equation}{Equation}{Equations}

\theoremstyle{plain}%
\newtheorem{theorem}{Theorem}[section]%

\newtheorem{proposition}[theorem]{Proposition}%
\newtheorem{lemma}[theorem]{Lemma}%
\newtheorem{corollary}[theorem]{Corollary}%
\newtheorem{definition}[theorem]{Definition}%

\newcommand{\proofof}[1]{\ \\\noindent \emph{Proof of {#1}.}}
\newcommand{\Endproofof}{\hfill\qed\ \\\ \\}

\DeclareMathOperator{\Str}{Str}
\DeclareMathOperator{\Th}{Th}
\DeclareMathOperator{\Aut}{Aut}

\DeclareMathOperator{\dcl}{dcl}
\DeclareMathOperator{\acl}{acl}

\newcommand{\id}{\mathrm{id}}

\newcommand{\ELA}{\mathrm{E}_{L,A}}

\newcommand{\weakconv}{{ \xrightarrow{\mathrm{w}}}}

\renewcommand{\hat}{\widehat}

\newcommand{\llrr}[1]{{\llbracket #1 \rrbracket}}
\newcommand{\bigllrr}[1]{{\bigl \llbracket #1 \bigr \rrbracket}}
\newcommand{\Bigllrr}[1]{{\Bigl \llbracket #1 \Bigr \rrbracket}}

\def\st{\,:\,}

\def\M{{\EM{\mathcal{M}}}}

\def\cM{{\EM{\mathcal{M}}}}
\def\N{{\EM{\mathcal{N}}}}
\def\cN{{\EM{\mathcal{N}}}}
\def\cK{{\EM{\mathcal{K}}}}

\def\AA{{\EM{\mathcal{A}}}}
\def\BB{{\EM{\mathcal{B}}}}
\def\PP{{\EM{\mathcal{P}}}}

\def\cK{{\EM{\mathcal{K}}}}

\def\F{{\EM{\mc{F}}}}

\def\x{{\EM{\ol{x}}}}
\def\xx{{\EM{\ol{x}}}}
\def\y{{\EM{\ol{y}}}}

\def\uu{{\EM{\ol{u}}}}

\def\zzi{{\EM{\ol{z_i}}}}
\def\zzj{{\EM{\ol{z_j}}}}

\def\a{{\EM{\ol{a}}}}

\def\aa{{\EM{\ol{a}}}}

\def\ww{{\EM{\ol{w}}}}
\def\wwz{{\EM{\ol{w_0}}}}
\def\wwo{{\EM{\ol{w_1}}}}
\def\wwi{{\EM{\ol{w_i}}}}
\def\uu{{\EM{\ol{u}}}}

\newcommand{\tf}{t_\mathrm{full}}

\def\Lomega#1{{\EM{\mc{L}_{#1, \w}}}}
\def\Lww{\Lomega{\w}}
\def\Lwow{\Lomega{\w_1}}

\newcommand{\defn}[1]{{\bf{#1}}}
\newcommand{\defas}{{\EM{\ :=\ }}}

\def\w{\EM{\omega}}

\def\Naturals{{\EM{{\mbb{N}}}}}
\def\Nats{{\EM{{\mbb{N}}}}}

\def\^{\EM{{}^{\And}}}

\def\Or{\EM{\vee}}
\def\And{\EM{\wedge}}

\def\<{\EM{\langle}}
\def\>{\EM{\rangle}}

\def\EM#1{\ensuremath{#1}}
\def\mbb#1{\EM{\mathbb{#1}}}

\def\mc#1{\EM{\mathcal{#1}}}

\def\ol#1{\EM{\overline{#1}}}

\def\ul#1{\underline{#1}}

\newcommand{\sympar}[1]{\ensuremath{S_{#1}}}

\newcommand{\sym}{\sympar{\infty}}

\newcommand{\Hom}{\ensuremath{\mathrm{Hom}}}
\newcommand{\Full}{\ensuremath{\mathrm{Full}}}

\definecolor{MyGreen}{rgb}{.75,0,.75}
\definecolor{RealGreen}{rgb}{0,1,0}
\definecolor{DarkGreen}{rgb}{0.1,0.4,0.1}
\definecolor{ActualGreen}{rgb}{0.0,0.5,0.0}
\definecolor{MyBlue}{rgb}{0,0,1}
\definecolor{MyRed}{rgb}{1,0,0}

\definecolor{darkred}{rgb}{0.5,0,0}
\definecolor{darkgreen}{rgb}{0, 0.3,0}
\definecolor{darkblue}{rgb}{0,0,0.6}

\begin{document}

\title[Countable infinitary theories admitting an invariant measure]{Countable infinitary theories\\ admitting an invariant measure}

\author[Ackerman]{Nathanael Ackerman}
\address{
Department of Mathematics\\
Harvard University\\
One Oxford Street\\
Cambridge, MA 02138\\
USA}
\email{nate@math.harvard.edu}

\author[Freer]{Cameron Freer}
\address{
	Remine\\
	2722 Merrilee Drive Suite 300\\
	Fairfax, VA 22031\\
	USA
}
\email{cameron@remine.com}

\author[Patel]{Rehana Patel}
\address{Department of Mathematics and Computer Science\\
Wheaton College\\
Norton, MA\\
02766\\
USA}
\email{patel\_rehana@wheatoncollege.edu}


\begin{abstract}
	Let $L$ be a countable language. We characterize, in terms of definable closure, those countable theories
	$\Sigma$ of $\Lwow(L)$ for which there exists an $\sym$-invariant probability measure on the collection of 
	models of $\Sigma$ with underlying set $\Nats$.  Restricting to $\Lww(L)$,
	this answers an open question of Gaifman from 1964, via a translation between
	$\sym$-invariant measures and Gaifman's symmetric measure-models with strict equality.
	It also extends the known characterization in the case where $\Sigma$ implies a Scott sentence. 
	To establish our result, we introduce machinery for building
	invariant measures from a directed system of countable structures with measures.
\end{abstract}

\maketitle


\setcounter{page}{1}
\thispagestyle{empty}

\begin{small}
\begin{tiny}
\renewcommand\contentsname{\!\!\!\!}
\setcounter{tocdepth}{3}
\tableofcontents
\end{tiny}
\end{small}

\section{Introduction}
\label{sec:intro}

Logic and probability bear many formal resemblances, and there is a long history of model-theoretic approaches to unifying them.
A seminal work extending classical model theory to random structures
is Gaifman's 1964 paper
\cite{MR0175755}.
This paper provides
coherence conditions for assigning probabilities to formulas from some first-order language,
instantiated from a fixed domain, in a way
that respects the logical relationships between formulas.
Gaifman calls such an assignment of probabilities a \emph{measure-model};
in the case where the assignment
also respects equality, he calls it a \emph{measure-model with strict equality}.

A key case that Gaifman addresses is that of a \emph{symmetric measure-model}, where 
the probabilities assigned to a formula are invariant under arbitrary finite permutations of the instantiating domain.
He shows how to obtain a symmetric measure-model for an arbitrary countable first-order theory, but demonstrates that some of these theories admit only symmetric measure-models without strict equality.
This leaves open the following question from
\cite[\S4]{MR0175755}:
\begin{quote}
	 ``The problem of characterizing those
	 theories (i.e.\ measures having the values $0$ and $1$) which
	 possess a measure-model with strict equality, satisfying also the
	 symmetry condition, seems to be \nobreak difficult.''
\end{quote}

Gaifman's paper is concerned with first-order theories of $\Lww$, but his question is also natural in the infinitary setting of $\Lwow$ explored by Scott and Krauss in \cite{ScottKrauss} and \cite{MR0275482}.
In the present paper, we answer the more general question 
for arbitrary countable theories of $\Lwow$
in the case of countable domains, 
by developing methods that extend our earlier work on invariant measures \cite{AFP}.
In doing so, we answer Gaifman's original question, for countable domains and languages.

Our setting is the following; for more details, see \S\cref{logicaction} and \ref{ergodic-structures}.
Let $L$ be a countable language, and write $\Str_L$ for the measurable space of $L$-structures with underlying set $\Nats$. The space $\sym$ of permutations of $\Nats$ acts on $\Str_L$ by permuting the underlying set.
Given a countable collection of sentences $\Sigma \subseteq \Lwow(L)$,
we consider when there is
an
$\sym$-invariant probability measure on $\Str_L$ that assigns probability $1$ to the class of models of $\Sigma$ in $\Str_L$.

For $L$ a countable language and $X$ a countable set,
the paper
\cite{MR0175755}
shows essentially that 
every measure-model with strict equality, of formulas of $\Lwow(L)$ instantiated by elements from $X$,
is induced by some
probability distribution
on the class of $L$-structures
with
underlying set $X$.
In the case where the underlying set $X$ is $\Nats$, 
Gaifman's symmetric measure-models with strict equality correspond
to $\sym$-invariant probability measures on $\Str_L$.
An $\Lww(L)$-theory $\Sigma$ 
has a symmetric measure-model with strict equality 
when the corresponding $\sym$-invariant probability measure assigns probability $1$ to the class of models of $\Sigma$ in $\Str_L$.

\subsection{Main result}
Our main theorem, \cref{main-result}, states that for a countable language $L$ and countable theory $\Sigma$ of $\Lwow(L)$, there is an $\sym$-invariant probability measure concentrated on the class of models of $\Sigma$ in $\Str_L$ precisely when $\Sigma$ has a completion (in some countable fragment) that has trivial definable closure (for that fragment) --- a criterion that is often easy to check in practice.

This theorem is a generalization of the main result of \cite{AFP}, which considered only the case where $\Sigma$ implies a Scott sentence. Recall that a Scott sentence is a sentence $\sigma\in \Lwow(L)$ that has exactly one countable model up to isomorphism; in particular, $\{\sigma\}$, though not deductively closed, is a complete theory of $\Lwow(L)$ in the sense that it implies $\varphi$ or implies $\neg \varphi$ for each sentence $\varphi\in\Lwow(L)$. \Cref{main-result} can therefore be viewed as a generalization of \cite{AFP} to countable theories of $\Lwow(L)$ that are not necessarily complete for $\Lwow(L)$.

For an $\sym$-invariant probability measure $\mu$ on $\Str_L$, we say that $\mu$ is \emph{ergodic} if every $\mu$-almost $\sym$-invariant subset of $\Str_L$ is assigned measure $0$ or $1$ by $\mu$.
Every $\sym$-invariant probability measure can be decomposed as a convex combination of ergodic ones. Further, if there is an $\sym$-invariant probability measure concentrated on a given Borel set $B\subseteq\Str_L$, then there is an ergodic such measure, and so it often suffices to consider only the ergodic $\sym$-invariant probability measures.

\begin{theorem}
	\label{main-result}
	Let $L$ be a countable language,
	and let $\Sigma \subseteq \Lwow(L)$ be a countable 
	set of sentences.
	Then the following are equivalent:

\begin{enumerate}
\item There is an $\sym$-invariant probability measure concentrated on the class of models of $\Sigma$ in $\Str_L$.

\item There is an ergodic $\sym$-invariant probability measure concentrated on the class of models of $\Sigma$ in $\Str_L$.

\item There is a countable fragment $A \subseteq \Lwow(L)$ and a consistent theory 
		$T\subseteq A$ that is complete for $A$, such that  $\Sigma \subseteq T$
		and $T$ has trivial $A$-definable closure. 

\item For all countable fragments $A \subseteq \Lwow(L)$ such that $\Sigma \subseteq A$, there is a 
	consistent theory 
		$T \subseteq A$ that is complete for $A$, such that $\Sigma \subseteq T$ 
		and $T$ has trivial $A$-definable closure. 
\end{enumerate}
\end{theorem}

Note that when $\Sigma$ is itself 
complete 
for some countable fragment $A\subseteq \Lwow(L)$, then 
	\cref{main-result} implies that there is an $\sym$-invariant probability measure concentrated on the class of models of $\Sigma$ 
	in $\Str_L$
	if and only if $\Sigma$ has trivial $A$-definable closure.

In the case where $\Sigma$ is a set of first-order sentences,
our main result simplifies to the following. 

\begin{theorem}
	\label{main-result-FO}
	Let $\Sigma \subseteq \Lww(L)$ be a set of first-order sentences.
	Then there is an $\sym$-invariant probability measure concentrated on the class of models of $\Sigma$ in $\Str_L$ if and only if there is some complete consistent theory 
	$T\subseteq\Lww(L)$ with $\Sigma \subseteq T$ such that $T$ has trivial $\Lww(L)$-definable closure.
\end{theorem}

This answers Gaifman's question for countable languages and domains.

\subsection{Outline of the paper}
We begin, in 
\cref{sec:preliminaries},
by providing
definitions regarding our setting, including the notion of an \emph{ergodic structure}, i.e., an ergodic $\sym$-invariant probability measure on $\Str_L$.
We describe
a process we call \emph{pithy $\Pi_2$ Morleyization} that allows us to work with languages and theories that
have 
nice properties which allow us to carry out the main construction of our paper. We also provide basic results about trivial definable closure and its relation to duplication of formulas.

In 
\cref{inv-meas-from-str},
we describe a method for building ergodic structures via sampling from a Borel structure equipped with a measure.

Next, in Sections \ref{sec:layerings} and \ref{sec:existence},
a countable consistent theory that is complete for a countable fragment and has trivial definable closure for that fragment, 
we show how to build a Borel structure and measure such that sampling from it yields 
an \emph{ergodic model} of the theory, i.e.,
an ergodic structure that concentrates on the collection of models of the theory.
We do so by building a special type of directed system of finite structures with measures whose limit is such a Borel structure along with a measure.

In 
\cref{nonexistence},
we show that if a countable consistent theory that is complete for some countable fragment has non-trivial definable closure for that fragment,
then it does not admit an ergodic model.

Finally, in 
\cref{classification},
we combine the positive and negative results from 
Sections \ref{sec:existence} and \ref{nonexistence}
to obtain our main result, \cref{main-result}, and its corollary for first-order languages, \cref{main-result-FO}.

\section{Preliminaries}
\label{sec:preliminaries}

In this section, we introduce and develop notions involving invariant measures and certain kinds of theories that will be useful in our constructions.

Throughout this paper, let $L$ be a countable language. 
We allow relation symbols to be $0$-ary. The instantiation of a $0$-ary relation in an $L$-structure is the assignment True or False; such relations allow us to simplify the Morleyization construction in \cref{pithy}.
For a nested pair of countable languages $L' \subseteq L''$ and an $L''$-structure $\N$, we write $\N|_{L'}$ to denote the reduct of $\N$ to the language $L'$.

We will formally allow formulas only to contain $\And, \bigwedge, \neg, \exists$, and we use $\Or, \bigvee, \forall$ as shorthand in the standard way. This simplification loses no generality, but sometimes allows us a cleaner presentation, as it reduces the number of cases to deal with for inductions on formulas.

Recall that a \defn{fragment} (of $\Lwow(L)$) is a subset $A \subseteq \Lwow(L)$ that is closed under subformulas as well as the logical operations of $\And, \Or, \neg, (\exists x)$ and $(\forall x)$. 

We say that a set of sentences $T \subseteq \Lwow(L)$ is a \defn{theory} when it is consistent.
In general, we do not require theories to be deductively closed or complete (for either a countable fragment or for $\Lwow(L)$).

Let $A$ be a fragment.  An \defn{$A$-theory $T$} is a theory
$T$ such that $T \subseteq A$. 
An $A$-theory $T$ is \defn{complete for $A$} when for every sentence $\sigma \in A$, either
$T\models \sigma$ or $T\models \neg \sigma$; in this case we say that $T$ is a \defn{complete $A$-theory}.
Note that we do not require even complete $A$-theories to be deductively closed; this will be important
when we 
work with $A$-theories 
of a restricted syntactic form (namely, pithy $\Pi_2$, as described in \cref{pithy}).

For a measure $m$ and singleton set $\{x\}$, we often abbreviate $m(\{x\})$ by the notation $m(x)$,
and for a function $i$ we similarly abbreviate the inverse image $i^{-1}(\{x\})$ by $i^{-1}(x)$.
Likewise, we write $Y \cup x$ to denote $Y \cup \{x\}$ and $Y \setminus x$ to denote $Y \setminus \{x\}$ 
We write \emph{qf-type} to mean  a quantifier-free type.


\subsection{The logic action on the measurable space $\Str_L$}
\label{logicaction}
The measurable space $\Str_L$ 
has underlying space the collection of $L$-structures with underlying set $\Naturals$, and
$\sigma$-algebra generated by subbasic open sets of the form
\[
\llrr{\varphi(n_1, \dots, n_j)} \defas
 \{\M \in \Str_L: \M\models \varphi(n_1, \dots, n_j)\},
\]
where $\varphi$ is an atomic $L$-formula, $j$ is its number of free variables, and $n_1, \ldots, n_j \in \Naturals$. 
In fact, $\llrr{\varphi(n_1, \ldots, n_j)}$ is Borel for arbitrary formulas $\varphi$ of $\Lwow(L)$, by \cite[Proposition~16.7]{MR1321597}.  
Given a countable theory $T$, we write $\llrr{T}$ to mean $\llrr{\bigwedge_{\varphi \in T}\varphi}$.

Let $\sym$ denote the permutation group of $\Naturals$.  The \emph{logic action} of $\sym$ on $\Str_L$ is the action induced by permutation of the underlying set $\Naturals$;  for more details, see \cite[\S16.C]{MR1321597}.  
We say that a Borel probability measure $m$ on $\Str_L$ is \defn{invariant} when it is invariant under the action of $\sym$, i.e., 
$m(X) = m(g\cdot X)$
for every Borel set $X \subseteq \Str_L$ and $g \in \sym$; we often simply call such an $m$ an \defn{invariant measure}.
A probability measure $m$ is \defn{concentrated} on a Borel set $X \subseteq \Str_L$ when $m(X) = 1$.

The Lopez-Escobar theorem 
(see, e.g., \cite[Theorem~16.8]{MR1321597}) states 
that a set $X\subseteq \Str_L$ is Borel and invariant under the logic action if and only if there is a sentence $\varphi \in \Lwow(L)$ such that $X = \llrr{\varphi}$.

We will sometimes speak of
an invariant measure concentrated on a sentence $\varphi$, by which we mean an invariant measure concentrated on $\llrr{\varphi}$, 
the class of models of $\varphi$ in $\Str_L$.
We will likewise speak of an invariant measure concentrated on a countable theory $T$, by which we mean an invariant measure concentrated on $\llrr{T}$, 
the class of models of $T$ in $\Str_L$,
and in this case say that $T$ \defn{admits} an invariant measure.
Given a countable collection $\Theta$ of qf-types, we say that an invariant measure \defn{omits $\Theta$} when it is concentrated on the class of structures in $\Str_L$ that omit every qf-type in $\Theta$, i.e., on 
	$ \bigllrr{\bigwedge_{p(\x)\in \Theta} (\forall \x) \neg p(\x)}$.

\subsection{Ergodic structures}
\label{ergodic-structures}

In order to determine which countable theories admit an invariant measure, it will suffice to ask which admit an 
ergodic invariant measure.

\begin{definition}
We say that a probability measure $\mu$ on $\Str_L$ is \defn{ergodic} if $\mu(B) \in \{0, 1\}$ for all Borel sets $B$ that satisfy 
$\mu(B \Delta \tau^{-1}(B)) = 0$ for every $\tau\in \sym$.
\end{definition}
In other words, an ergodic probability measure $\mu$ is one that does not assign intermediate measure to any $\mu$-almost invariant set.

We will use the following standard result in the proof of \cref{main-result}; we include its proof for completeness.

\begin{lemma}
	If $\Sigma\subseteq \Lwow(L)$ is countable set of sentences that admits an $\sym$-invariant measure, then $\Sigma$ admits an ergodic $\sym$-invariant measure.
\label{ergodic-lemma}
\end{lemma}
\begin{proof}
Let $\mu$ be an invariant measure concentrated on $\Sigma$,
and suppose that $\Sigma$ does not admit an ergodic invariant measure.
	The measure $\mu$ can be decomposed into a mixture of ergodic invariant measures on $\Str_L$ (see, e.g., \cite[Lemma~A1.2 and Theorem~A1.3]{MR2161313}).
	By hypothesis, none of the measures in the decomposition is concentrated on $\Sigma$. 
	By the Lopez-Escobar theorem, $\llrr{\Sigma}$ is Borel invariant, and so
	the measures in the decomposition must therefore all be concentrated on the complement of $\llrr{\Sigma}$.
	But then $\mu$, which is a mixture of the measures in the decomposition, is concentrated on the complement of $\llrr{\Sigma}$, a contradiction.
\end{proof}

Ergodic invariant measures can be thought of as ``probabilistic'' structures, as we now describe.

\begin{definition}
	An \defn{ergodic $L$-structure} is an ergodic invariant measure on $\Str_L$.
	An ergodic $L$-structure $\mu$ is said to \defn{almost surely satisfy} a sentence $\varphi\in\Lwow(L)$
	when $\mu(\llrr{\varphi}) = 1$.
\end{definition}

\begin{lemma}
	\label{complete-consistent}
	Let 
	$\mu$ be an ergodic $L$-structure, and define
	\[
		\Th(\mu) \defas \{\varphi\in \Lwow(L) \st \mu(\llrr{\varphi})=1\}.
	\]
	Then $\Th(\mu)$
	is a complete,
	deductively closed
	$\Lwow(L)$-theory.
\end{lemma}
\begin{proof}
	Because the measure $\mu$ is ergodic, for any sentence $\varphi\in\Lwow(L)$,
	exactly one of $\varphi\in \Th(\mu)$ or $\neg \varphi\in \Th(\mu)$ holds, by the Lopez-Escobar theorem.
	Deductive closure follows from $\sigma$-additivity, and so $\Th(\mu)$ is also consistent, hence an $\Lwow(L)$-theory.
\end{proof}

\begin{definition}
	We say that $\mu$ is an \defn{ergodic model} of $T$
	when $T \subseteq \Th(\mu)$; in that case we say that $\Th(\mu)$ is the theory of $\mu$.
\end{definition}

An ergodic structure therefore has a complete $\Lwow(L)$-theory, and it is an ergodic model of this theory.
This provides some justification for considering ergodic structures as probabilistic generalizations of classical model-theoretic structures.

Finally, we say that an ergodic $L$-structure almost surely has a property $P$ when it assigns measure $1$ to the collection of elements of $\Str_L$ having property $P$. For example, we say that $\mu$ almost surely omits a countable collection of qf-types $\Theta$ when 
	$\mu\big(\bigllrr{\bigwedge_{p(\x)\in \Theta} (\forall \x) \neg p(\x)}\bigr) = 1$.
We will be especially interested in ergodic models of a given theory that almost surely omit a particular collection of qf-types.

\subsection{Morleyization and Pithy $\Pi_2$ theories}
\label{pithy}

In our main construction it will be important to work with a first-order theory consisting of sentences of a specific form which can be thought of as ``one-point extension axioms''. It will also be important that our measure concentrate on the collection of models of a theory of one-point extension axioms that omits a countable collection of (non-principal) qf-types, which we obtain by a variant of a standard construction.

The notion of \emph{non-redundant} tuples, formulas, structures, and theories will be important 
when formulating the notions of \emph{duplication of qf-types}
(\cref{newstrongamalgamation}) and \emph{layering transformations}
(\cref{transformation}).
\begin{definition}
	\label{nonredundant-def}
	A tuple $a_1 \cdots a_n$ is 
	\defn{non-redundant} if $a_i \neq a_j$ for $i\neq j$.
A formula with free variables $x_1, \ldots, x_n$
	is 
	\defn{non-redundant} if it implies the formula $\bigwedge_{1\le i< j\le n} (x_i \neq x_j)$.
	A structure in a relational language
	is \defn{non-redundant} if each relation holds only on non-redundant tuples, and a theory
	in a relational language
	is \defn{non-redundant} if all of its models are non-redundant.
\end{definition}

We now introduce the special form of sentences, which we call pithy $\Pi_2$. We then show that for any 
countable fragment $A$ of
$\Lwow(L)$, there is a relational language $L_A$, a pithy $\Pi_2$ first-order $L_A$-theory $\Th_A$,
and a countable collection of qf-types $\Theta_A$ of $L_A$, such that 
each $L$-structure has a common definable expansion with some
model of $\Th_A$ omitting $\Theta_A$.

\begin{definition}
A sentence $\varphi \in \Lwow(L)$ is said to be \defn{pithy $\Pi_2$} if it is of the form $(\forall\xx)(\exists y)\psi(\xx, y)$ where $\psi(\xx,y)$ is quantifier-free, and $\xx$ is a finite (possibly empty) sequence of variables. A countable theory $T$ is said to be pithy $\Pi_2$ when every sentence in $T$ is pithy $\Pi_2$.  
\end{definition}

Throughout the rest of this subsection, let $A$ be a countable fragment of $\Lwow(L)$ and let $T$ be an $A$-theory.

We now define the relational language $L_A$. 
For $m\in\Nats$ write $[m] \defas \{1, \ldots, m\}$.
Let $I$ be the set of triples $( \varphi(\x), \iota, \kappa)$ for which there exist $n\in \Nats$ and $r\le n$ such that:
\begin{itemize}
	\item $\varphi(\x) \in A$ is an $L$-formula with $\x = x_1 \cdots x_n$ its tuple of distinct free variables,
	\item $\iota \colon [n] \to [r]$ is a surjection, and
	\item $\kappa \colon [r] \to [n]$ is an injection such that $\iota \circ \kappa$ is the identity on $[r]$.
\end{itemize}

Then define the relational language 
\[
	L_A \defas \{
		Q_{\varphi(\xx), \iota, \kappa} \ \st  \ (\varphi(\x), \iota, \kappa)\in I\},
\]
where each $Q_{\varphi(\xx), \iota, \kappa}$ is a relation symbol of arity $|\xx|$.

For each surjection $\iota\colon [n] \to [r]$, define
\[
	\mathrm{eq}_\iota \defas 
\bigwedge_{j, \ell \in [n]\st \iota(j) = \iota(\ell)} (x_j = x_\ell) 
\ \wedge
\ \bigwedge_{j, \ell \in [n] \st \iota(j) \neq \iota(\ell)} (x_j \neq x_\ell).
\]
Now for each  $\eta(\x) \in A$, 
define the formula
\[ R_{\eta(\x)}(\x) \defas \bigvee_{
	\iota, \kappa \st (\eta(\x), \iota, \kappa)\in I }
\{ 
\mathrm{eq}_\iota \wedge
	Q_{\eta(\x), \iota, \kappa}(x_{\kappa(1)}\cdots x_{\kappa(r)})
	\},\]
	where $r$ is the arity of $Q_{\eta(\x), \iota, \kappa}$ and $\x = x_1\cdots x_n$ is the tuple of distinct free variables of $\eta$.

After two observations, we next define the theory $\Th_A$ connecting $L_A$ to $L$ (which will be the pithy $\Pi_2$ Morleyization of the empty theory).
	First, note that we have described multiple relation symbols that our theory will prove equivalent; 
	namely,
	whenever
	$(\varphi(\x), \iota, \kappa), (\varphi(\x), \iota, \kappa')\in I$ we have
	\[ \Th_A \models  Q_{\varphi(\x), \iota, \kappa}(\y) \leftrightarrow
	Q_{\varphi(\x), \iota, \kappa'}(\y),
	\]  though this does not pose a problem (and simplifies naming of the $Q$ relation symbols).

	Second, note that the language $L_A$ is ``built from'' $L$ in the sense that for each atomic $L$-formula $\eta(\x)$ there is 
	a first-order quantifier-free $L_A$-formula $R_\eta(\x)$
	that 
	$\Th_A$ will prove equivalent to $\eta(\x)$.
	However, $L_A$
	does not literally contain $L$ as a sublanguage (and we will later refer to the language $L_A \cup L$ when needed).
Note that when $\varphi$ is a sentence, then each $Q_{(\varphi(\x), \iota, \kappa)}$ is a $0$-ary relation symbol, and 
$R_\varphi$ is a $0$-ary formula, whose presence simplifies the technicalities of the pithy $\Pi_2$ Morleyization.  

Define the pithy $\Pi_2$-theory $\Th_A$ to be the collection of sentences
\begin{itemize}

	\item $\displaystyle 
		(\forall \y)[Q_{(\zeta(\uu), \iota, \kappa)}(\y) \to \bigwedge_{j, \ell \in [k] \st j\neq \ell} (y_j \neq y_\ell)]$
\end{itemize}
		for all $(\zeta(\uu), \iota, \kappa)\in I$ where $\y = y_1\cdots y_k$ is the tuple of distinct free variables in $Q_{(\zeta(\uu), \iota, \kappa)}$, and
\begin{itemize}
\item $(\forall \xx)[R_{\neg \psi}(\xx) \leftrightarrow \neg R_\psi(\xx)]$,
\item $(\forall \xx)[R_{\xi_0 \And \xi_1}(\xx) \leftrightarrow \bigl(R_{\xi_0}(\wwz) \And R_{\xi_1}(\wwo)\bigr)]$,
\item $(\forall \xx)[R_{(\exists y)\varphi}(\xx) \leftrightarrow (\exists y) R_{\varphi}(\xx,y)]$,
\item $(\forall \xx)[R_{(\exists y)\psi}(\xx) \leftrightarrow (\exists y) R_{\psi}(\xx)]$, and
\item $(\forall \xx)[R_{\bigwedge_{i \in \w}\eta_i}(\xx) \rightarrow R_{\eta_j}(\zzj)]$ for all $j\in\omega$,
\end{itemize}
for all $\xi_0, \xi_1, \psi, \varphi, \bigwedge_{i \in \w}\eta_i\in A$; where $\xx$ is a tuple containing precisely the free variables of $\psi$; 
where the tuple $\wwi$ contains precisely the free variables of $\xi_i$ 
for $i = 0,1$ and $\xx$ contains precisely the variables occurring in $\wwz$ or $\wwo$,
where the tuple $\zzi$ contains precisely the free variables of $\eta_i$  for each $i \in \w$ and $\xx$ contains precisely the variables occurring in some $\zzi$; and where the free variables  of $\varphi$ are precisely the variables in $\xx y$, with $y\not\in\xx$.

We also define a theory $\ELA$ that will allow us to restrict to the case of relational languages in our main constructions. Let $\ELA$ be the $(L_A \cup L)$-theory consisting of
	$(\forall\ww) [R_{\zeta(\ww)}(\ww) \leftrightarrow \zeta(\ww)]$, 
		where $\zeta(\ww)$ is an atomic $L$-formula whose free variables are precisely those in the tuple $\ww$.

\begin{lemma}
	\label{nonredundant-ThA}
$\Th_A$ is a non-redundant pithy $\Pi_2$ first-order $L_A$-theory.
\end{lemma}
\begin{proof}
	$\Th_A$ is a first-order $L_A$-theory by definition, and is non-redundant by its first set of axioms (involving every relation symbol $Q_{(\zeta(\uu), \iota, \kappa)}$ in the language). The theory $\Th_A$ is also pithy $\Pi_2$ by the form of the remaining axioms and the fact that each formula $R_\zeta$ has a quantifier-free definition in terms of the $Q$ relations symbols.
\end{proof}

We now define a countable collection of qf-types $\Theta_A$ in the language $L_A$. 
For every formula of the form $\bigwedge_{i \in \w} \eta_i \in A$, define the qf-type
$p_{\bigwedge_{i \in \w} \eta_i}$ by
\[p_{\bigwedge_{i \in \w} \eta_i}(\xx) \defas \{\neg R_{\bigwedge_{i \in \w} \eta_i}(\xx)\} \cup 
\{R_{\eta_i}(\zzi) \st i \in \w\},\]
where $\xx$ and $\zzi$ are tuples of variables as above.
Then define the collection
\[
	\Theta_A \defas \{ p_{\bigwedge_{i \in \w} \eta_i}\st \bigwedge_{i \in \w} \eta_i \in A\}.\]
Note that if there is some $n\in\omega$ such that
\[
	\models (\forall \xx) \Bigl (\bigwedge_{i \in \w} \eta_i(\zzi) \leftrightarrow \bigwedge_{i \leq n} \eta_i(\zzi) \Bigr),
\]
then $p_{\bigwedge_{i \in \w} \eta_i}$ is inconsistent with $\Th_A \, \cup \, \ELA$. However, this is not a problem, as then the qf-type is automatically omitted.  

Define the $L_A$-theory $\Th_A^+$ to be 
\[
	\Th_A^+ \ \defas\ 	\Th_A \,\cup\,
		\{ (\forall \xx) \neg p(\xx) \st p(\xx) \in \Theta_A     \},
\]
	where the tuple of variables $\xx$ is of the appropriate length for each qf-type $p(\xx) \in \Theta_A$.
In fact, a model of $\Th_A^+$ is precisely the 
\emph{pithy $\Pi_2$
Morleyization} of some $L$-structure 
	(see \cref{Morley-def} below).

An easy induction on $L$-formulas shows 
that in any model $\cK$ of $\Th_A^+ \, \cup \, \ELA$, for all formulas $\psi \in A$, we have
\[  \cK \models (\forall \x) \bigl(\psi(\x) \leftrightarrow R_\psi(\x)\bigr),
	\]
	where $\x$ contains precisely the free variables of $\psi$, i.e., $\cK$ omits $\Theta_A$.
	In fact, we have the following.

\begin{lemma}
\label{Lemma: Pithy Pi-2 Morleyization}
	For every $L$-structure $\M$ there is a unique expansion of $\M$ to an $(L_A \cup L)$-structure $\M_A$ such that 
	$\M_A \models \Th_A^+ \, \cup \, \ELA$.
In addition, for every $L_A$-structure $\N$ that satisfies $\Th_A^+$,
	there is a unique 
	expansion of $\N$ to an $(L_A \cup L)$-structure $\N^\mathrm{E} \models \ELA$. 
	Further, $(\M_A|_{L_A})^\mathrm{E} = \M_A$, and $(\N^\mathrm{E}|_{L})_A = \N^\mathrm{E}$.
\end{lemma}
\begin{proof}
	Let $\M_A$ be the $(L_A\cup L)$-structure that satisfies $\Th^+_A  \, \cup \, \ELA$ 
	and whose reduct to $L$ is $\M$. Such a structure clearly exists and is the unique $(L_A\cup L)$-structure with these properties.
	 
	Likewise, let $\N^\mathrm{E}$ be the unique $(L_A\cup L)$-structure that satisfies $\Th^+_A  \, \cup \, \ELA$ 
	and whose reduct to $L_A$ is $\N$.
\end{proof}

\begin{definition}
	\label{Morley-def}
	Let $\M$ be an $L$-structure.
	The 
	\defn{pithy $\Pi_2$ Morleyization} of $\M$ for $A$ is the
	$L_A$-structure $\M_A|_{L_A}$.
\end{definition}
Observe that any pithy $\Pi_2$ Morleyization is in a relational language (by the construction of $L_A$) and is non-redundant (by \cref{nonredundant-ThA}).

Define the countable pithy $\Pi_2$ first-order $L_A$-theory
$T_A \defas \Th_A \,\cup\, \{R_\sigma \st \sigma \in T\}$, which 
we will use in \cref{sec:existence}.
Further define the 
$\Lwow(L_A)$-theory \linebreak
$T_A^+ \defas \Th_A^+ \,\cup\, \{R_\sigma \st \sigma \in T\}$,
and let $A^+$ be the smallest $L_A$-fragment containing $\Th_A^+$. Observe that $A^+$ is countable and $T_A^+$ is an $A^+$-theory; we will later see in 
	\cref{TAplus-complete-Aplus-theory}
that $T_A^+$ is a complete $A^+$-theory.

Note that an $L_A$-structure satisfies $T_A^+$ if and only if it satisfies $T_A$ and omits 
$\Theta_A$.
The theory $T_A^+$ has the following key property.
\begin{lemma}
	\label{theory-Morleyization-lift}
	Let $\M$ be an $L$-structure.
	For every $L$-sentence $\sigma \in A$, we have $\M \models \sigma$ if and only if $\M_A \models R_\sigma$. 
In particular, 
	$\M_A$ is a definable expansion of $\M$ for $A$, and
	$\M \models T$ if and only if $\M_A\models T_A^+$.
\end{lemma}
\begin{proof}
	We have
	$\M_A\models \Th_A^+ \, \cup \, \ELA$ 
by \cref{Lemma: Pithy Pi-2 Morleyization}. 
	Hence for atomic formulas $\sigma$, we have $\M \models \sigma$ if and only if $\M_A \models R_\sigma$.
	This equivalence holds for all $\sigma$
	via an easy induction on formulas, 
	using the way that $\Th_A^+$ was constructed.
\end{proof}

	For every $L$-type $p$ over $A$ that is consistent with $T$,
	define the qf $L_A$-type $p_A(\x) \defas \{ R_\psi(\x) \st \psi\in p\}$.
The following statement that types lift to qf-types 
is immediate from \cref{theory-Morleyization-lift}. 

\begin{corollary}
	\label{lift-to-qftypes}
	Let $p$ be a $L$-type over $A$ that is consistent with $T$.
Then $p_A$  is
consistent with $\Th^+_A$  
and such that 
$\M \models p(\a)$ iff $\M_A \models p_A(\a)$ for every model $\M$ of $T$ and tuple $\a\in\M$ whose length is the number of free variables of $p_A$.
\end{corollary}

We now establish a tight connection between models of $T$ and of $T_A^+$.

\begin{corollary}
	\label{borel-bijection}
	The map
	$\M \mapsto \M_A|_{L_A}$ restricts to a Borel bijection 
	between 
$\llrr{T}$ and $\llrr{T_A^+}$
	that commutes with the logic action and preserves convex combinations of probability measures.
\end{corollary}
\begin{proof}
By \cref{Lemma: Pithy Pi-2 Morleyization} and \cref{lift-to-qftypes},
the map $\M \mapsto \M_A|_{L_A}$ is a bijection between $\Str_L$ and the set of models of
$\Th_A^+$ within
$\Str_{L_A}$.
This bijection is clearly Borel; note that it commutes with the logic action and preserves convex combinations of probability measures. Further, by \cref{theory-Morleyization-lift},
the image of $\llrr{T}$ is $\llrr{T_A^+}$.
	Finally, note that every model of $T_A^+$ is the 
	pithy $\Pi_2$ Morleyization of some model of $T$.
\end{proof}

This yields the following corollary for ergodic structures, since the ergodic structures are precisely those invariant measures that cannot be decomposed as a non-trivial convex combination of others
	(again see \cite[Lemma~A1.2 and Theorem~A1.3]{MR2161313}).

\begin{corollary}
\label{bijection-preserves-invariant-measures}
$T$ has an ergodic model 
	if and only if
	$T_A^+$ does.
\end{corollary}
\begin{proof}
	By \cref{borel-bijection}, there is a bijection between the set of ergodic $L$-structures that almost surely satisfy $T$ and the ergodic $L_A$-structures that almost surely satisfy $T_A^+$.
	In particular, one such set of ergodic structures is non-empty if and only if the other is non-empty.
\end{proof}

This corollary reduces the problem of determining whether there is an ergodic 
model of
an $A$-theory to that of determining whether there is 
an ergodic model
of a corresponding first-order theory that omits a certain countable set of qf-types.

The following result will be useful when considering trivial $A$-definable closure in \cref{newstrongamalgamation}.
\begin{corollary}
	\label{lifting-of-provability}
		For any sentence $\sigma\in A$,  we have 
		\[T\models \sigma \qquad \text{if and only if} \qquad T_A^+ \models R_\sigma.
		\]
\end{corollary}
\begin{proof}
	By the completeness theorem for sentences of $\Lwow(L)$, we have $T\models \sigma$  if and only if 
	$\llrr{T \cup \{\neg \sigma\}}$ is empty. By 
	\cref{bijection-preserves-invariant-measures}, $\llrr{T \cup \{\neg \sigma\}}$ is empty if and only if
		$\llrr{T_A^+ \cup \{R_{\neg \sigma}\}}$ is empty.

Again by completeness,	
		$T_A^+ \models R_\sigma$ if and only if
	$\llrr{T_A^+ \cup \{\neg R_\sigma\}}$ is empty. 
But	
	\[
		\llrr{T_A^+ \cup \{\neg R_\sigma\}}
		= \llrr{(T_A^+ \cup \{R_{\neg \sigma}\}) },
	\]
	and so we are done.
\end{proof}

The following corollary will be used in the proof of our main result, \cref{main-result}.

\begin{corollary}
	\label{TAplus-complete-Aplus-theory}
For any complete $A$-theory $T$, the theory $T_A^+$ is a complete $A^+$-theory.
\end{corollary}
\begin{proof}
By an easy induction on formulas, for every formula $\varphi\in A^+$, there is a relation
$R\in L_A$, of arity the number of free variables of $\varphi$,
such that $\Th_A^+ \models (\forall \x) \bigl(R(\x) \leftrightarrow \varphi(\x)\bigr)$.
Therefore the result follows by \cref{lifting-of-provability}.
\end{proof}

\subsection{Definable closure and duplication of formulas}
\label{newstrongamalgamation}

The main result of \cite{AFP} is that the only obstacle to the existence of an 
ergodic structure concentrated on the isomorphism class of a given structure $\M$
is the presence of a tuple whose definable closure in $\M$ includes additional elements. The main result of the present paper shows that the only obstruction to 
the existence of an ergodic structure concentrated on the collection of models of
a complete $A$-theory $T \subseteq \Lwow(L)$ is the presence of a formula that $T$ proves \emph{uniformly} witnesses the non-triviality of definable closure in the models of $T$. 

In this subsection we introduce the notion of \emph{trivial $A$-definable closure} over a fragment $A$ as well as a useful equivalent concept, that of having \emph{duplication of formulas} in $A$.

\begin{definition}
	Let $A$ be a fragment and let $\M$ be an $L$-structure.
Let $T^\M$ be the complete $A$-theory of $\M$.
	The \defn{$A$-definable closure} of a set $X \subseteq \M$ is defined to be
	the set $\dcl_A(X)$ of all elements $b\in \M$ for which there is a formula $\varphi(\x,y)\in A$ and a tuple $\a$ of length $|\x|$ all of whose elements are in $X$, such that
	\[
		\M\models \varphi(\a, b) \qquad \mathrm{and} \qquad
	T^\M \models (\forall \x)(\exists^{=1} y)\, \varphi(\x,y),\]
where $(\exists^{=1} y)$ means ``there exists a unique $y$''.
	The \defn{$A$-algebraic closure} of $X$, written $\acl_A(X)$ is defined similarly except with $(\exists^{=1} y)$ replaced by $(\exists^{< \omega} y)$, i.e., ``there exist finitely many $y$''.
\end{definition}

Recall that the Scott sentence $\sigma_\cM\in\Lwow(L)$ of a countable $L$-structure $\cM$ determines $\cM$ up to isomorphism among countable structures.
When $A$ contains the Scott sentence of $\M$, then $\dcl_A(X)$ is the same as the usual ``group-theoretic'' notion of definable closure in $\M$, i.e., the set $\dcl^{\textrm{G}}_\cM(X)$ of those $b\in\M$ for which 
\[\bigl\{g(b) \st g \in \Aut(\cM)\text{~s.t.~}(\forall c\in X)\ g(c) = c\bigr\} = \{b\}.\]
The structure $\cM$ is said to have trivial group-theoretic definable closure when $\dcl^{\textrm{G}}_\cM(X) = X$ for all finite $X\subseteq\cM$.
Indeed, this is the notion of definable closure used in \cite{AFP}, which only considers the case of Scott sentences.

\begin{definition}
	Let $A$ be a fragment and let $\M$ be an $L$-structure.
The structure $\M$ has	
	\defn{trivial $A$-definable closure} when 
	$\dcl_A(X) = X$ for all $X \subseteq\M$. 
	Note that $\M$ already has trivial $A$-definable closure
if 	
	$\dcl_A(X) = X$ for all \emph{finite} $X \subseteq\M$. 
Also note that the analogously defined notion of trivial $A$-algebraic closure is identical to trivial $A$-definable closure.

	A complete $A$-theory $T$ has \defn{trivial $A$-definable closure} when all models $\M\models T$ have trivial $A$-definable closure. 
\end{definition}

One can also perform a formula-by-formula analysis of trivial $A$-definable closure.

\begin{definition}
	Let $A$ be a countable fragment of $\Lwow(L)$.
A complete $A$-theory $T$ has \emph{trivial $\varphi(\x, y)$-definable closure} for a formula $\varphi(\x, y)$ of $\Lwow(L)$ if $T \models \neg (\exists \x)(\exists^{=1} y)\,  \varphi(\x, y)$. 
\end{definition}
The following lemma is immediate.

\begin{lemma}
	\label{trivial-dcl-formulas}
	Let $A$ be a countable fragment of $\Lwow(L)$.
	A complete $A$-theory $T$ has trivial $A$-definable closure if and only if $T$ has trivial $\psi$-definable closure for all non-redundant formulas $\psi\in A$.
\end{lemma}

It was shown in \cite{AFP} that $\cM$ has trivial group-theoretic definable closure if and only if there is an ergodic structure concentrated on its orbit.
When we move to $A$-theories it is not the case that having an 
ergodic model of
the theory ensures that all (classical) models of the theory have trivial group-theoretic definable closure. In fact, \cite{AFNP} showed that there are many first-order theories on which some invariant measure is concentrated, but for which almost every (classical) structure sampled according to the measure 
has non-trivial group-theoretic definable closure. 
For further examples of this phenomenon, see \cite{properly-ergodic}.

	\begin{lemma}
		\label{bijection-preserves-trivial-dcl}
	Let $A$ be a countable fragment of $\Lwow(L)$, and let $T$ be a complete $A$-theory.
		Then $T$ has trivial $A$-definable closure if and only if $T^+_A$ has trivial $A^+$-definable closure.
	\end{lemma}
	\begin{proof}
		For every sentence $\sigma \in A^+$, there is a relation symbol $R\in L_A$ such that 
		\[T^+_A \models \sigma \leftrightarrow R.\]
However, by \cref{lifting-of-provability},
		for every formula $\varphi\in A$,  the theory
$T$ has trivial $\varphi$-definable closure if and only if 
		$T^+_A$ has trivial $R_\varphi$-definable closure.

		In particular, $T$ has trivial $\varphi$-definable closure for all formulas $\varphi\in A$ if and only if $T^+_A$ has trivial $\psi$-definable closure for all formulas $\psi\in A^+$.
		Therefore, by \cref{trivial-dcl-formulas}, $T$ has trivial $A$-definable closure
		if and only if $T^+_A$ has trivial $A^+$-definable closure.
	\end{proof}

	Rather than working with the notion of trivial $A$-definable closure directly,
	it will sometimes be convenient to use
the equivalent property of 
having \emph{duplication of formulas} in $A$, which is closely related to a notion that first appeared in \cite[\S3.3]{AFP}.
Recall the notion of a non-redundant formula (\cref{nonredundant-def}).

\begin{definition}
\label{dupdef}
	Let $T$ be a theory and $\varphi(\x, y) \in \Lwow(L)$ be a non-redundant formula consistent with $T$.
The theory $T$ \defn{duplicates} $\varphi$ when 
	\[
		T \models (\exists \x,y,z) \, \bigl(
		\varphi(\x, y) \,\And\, \varphi(\x, z) \, \And\, (y \neq z)\bigr).
	\]
	The theory
	$T$ has \defn{duplication of quantifier-free formulas} when it duplicates every non-redundant quantifier-free first-order formula $\varphi(\x, y)\in \Lww(L)$ consistent with $T$.

Let $A$ be a fragment and suppose that $T$ is an $A$-theory.
The theory
	$T$ has \defn{duplication of formulas in $A$} when $T$ duplicates every non-redundant formula 
 in $A$
	consistent with $T$.
\end{definition}

We then have the following. 
\begin{lemma}
\label{multiple-duplication}
Let $A$ be a fragment, and
suppose that $T$ has duplication of formulas in $A$.
For any non-redundant formula $\varphi(x_1, \dots, x_k) \in A$ such that
$(\exists x_1, \dots, x_k)\,\varphi(x_1, \dots, x_k)$ is consistent with $T$ and any
	$\ell\le k$, there is a non-redundant formula
$\psi(x_1^0,x_1^1, \dots, x_\ell^0, x_\ell^1)$ such that
\begin{itemize}
	\item $(\exists x_1^0,x_1^1, \dots, x_\ell^0, x_\ell^1, x_{\ell+1}, \ldots, x_k)\,\psi(x_1^0,x_1^1, \dots, x_\ell^0, x_\ell^1, x_{\ell+1}, \ldots, x_k)$ is consistent with $T$. 

\item For all maps $\alpha \colon \{1, \dots, \ell\} \rightarrow \{0,1\}$,
	\begin{eqnarray*}
		T\models (\forall x_1^0,x_1^1, \dots, x_\ell^0, x_\ell^1, x_{\ell+1}, \ldots, x_k)
		\hspace*{150pt} \\
		\hspace*{20pt}
		\bigl( \psi(x_1^0,x_1^1, \dots, x_\ell^0, x_\ell^1, x_{\ell+1}, \ldots, x_k)\rightarrow \varphi(x_1^{\alpha(1)}, \dots,x_\ell^{\alpha(\ell)}, x_{\ell+1}, \ldots, x_k) \bigr).
		\hspace*{-45pt}
	\end{eqnarray*}
\end{itemize}

\end{lemma}
\begin{proof}
	The case where $\ell=1$ is immediate from the definition of duplication of formulas.
	For $\ell\ge 1$, the result follows from an easy induction.
\end{proof}

In our main construction, where we assume trivial definable closure, we will actually use the (potentially weaker) notion of duplication of formulas.

\begin{lemma}
\label{trivial dcl implies duplications}
Let $A$ be a fragment and let
	$T$ be a complete $A$-theory. 
Suppose that $T$ has trivial $A$-definable closure. Then
$T$ has duplication of formulas in $A$.
\end{lemma}
\begin{proof}
	Let $\varphi(\x, y)$ be an $A$-formula consistent with $T$.
	Then $T \models (\exists \x, y)\,\varphi(\x, y)$.
	But $T$ has trivial $A$-definable closure, and so by
	\cref{trivial-dcl-formulas}  we have
	$T \models \neg  (\exists \x) (\exists ^{=1} y)\,\varphi(\x, y)$.
Therefore	
	$T \models (\exists \x, y, z)\,\bigl(\varphi(\x, y) \And \varphi(\x, z) \And (y \neq z)\bigr)$.
\end{proof}
Note that the other direction of 
\cref{trivial dcl implies duplications}
does not necessarily hold.


\begin{lemma}
	\label{nonredundant-nontrivial-dcl}
	Let $T$ be a complete $A$-theory, and suppose that  $T$ has non-trivial $A$-definable closure.
	Then $T$ has non-trivial $\varphi(\x,y)$-definable closure for some non-redundant $\varphi(\x, y) \in A$, i.e.,
\[
	T \models (\exists \xx) (\exists^{=1}y)\, \varphi(\xx, y).
\]
\end{lemma}
\begin{proof}
	Because $T$ is a complete $A$-theory, if $T$ has non-trivial $\bigvee_{i\le n}\psi_i$-definable closure for some $\bigvee_{i\le n}\psi_i \in A$, then $T$ has non-trivial $\psi_i$-definable closure for some $i\le n$.
	But every formula in $A$ is equivalent to a disjunct of finitely many non-redundant formulas in $A$, and so the result follows.
\end{proof}

\section{Ergodic structures via sampling from Borel structures}
\label{inv-meas-from-str}

In this section we describe a general framework for constructing ergodic structures by sampling from continuum-sized structures, which
dates back to work of Aldous \cite{MR637937} and Hoover \cite{Hoover79}, and which has been used more recently in \cite{MR2724668}, \cite{AFP}, \cite{AFNP}, \cite{AFKP}, and \cite{properly-ergodic}. 

We begin by defining Borel $L$-structures, and then describe how to obtain an ergodic structure 
via a sampling procedure from a Borel $L$-structure.
We then describe weighted homomorphism densities, which will allow us to show convergence of finite sampled substructures to the desired measure on countably infinite structures.

\subsection{Borel $L$-structures}

When considering whether a given theory has an ergodic model,
without loss of generality we may restrict to the case of relational languages (by 
\cref{bijection-preserves-invariant-measures},
as $T^+_A$ is in a relational language).

Throughout this subsection, $L$ will denote a countable relational language.

\begin{definition}
\label{BorelLstructure}
Let $\PP$ be an $L$-structure.
We say that $\PP$ is a \defn{Borel $L$-structure} if 
there is a Borel $\sigma$-algebra on 
its 
underlying set $\BB_\PP$
such that
for all relation symbols $R \in L$,
the set
$\{\aa \in \PP^j \st R^\PP(\aa)\}$
is a Borel subset of $\BB_\PP^j$, where $j$ is the arity of $R$.
\end{definition}

We next describe a map taking an element of $\cN^\omega$ to an
$L$-structure with underlying set $\Naturals$. 
The application of this map to an appropriate random sequence of elements of
a Borel $L$-structure will induce a random $L$-structure with underlying set $\Naturals$.

\begin{definition}
Let $\cN$ be an $L$-structure (of arbitrary cardinality).
Define the function $\F_{\cN}  \colon \cN^\omega \to \Str_L$  as follows.
For $\AA=(a_i)_{i\in \w} \in \cN^\omega$, let $\F_{\cN}(\AA)$ be the $L$-structure with underlying set $\Naturals$ satisfying
\[
\F_{\cN}(\AA) \, \models \,R(n_1, \dots, n_j) \quad \Leftrightarrow \quad 
\cN\models R(a_{n_1}, \dots, a_{n_j})
\]
for every relation symbol $R \in L$ and for all $n_1, \dots, n_j \in \Naturals$, where $j$ is the arity of $R$.
\end{definition}
Note that we consider equality as a logical symbol, not a relation symbol (so that equality is inherited from the underlying set.

Observe that when
$\PP$ is a Borel $L$-structure, 
$\F_{\PP}$ is a Borel measurable function 
(see \cite[Lemma~3.3]{AFP}).

\begin{definition}
\label{mupm}
Let $\PP$ be a Borel $L$-structure,
and let $m$ be a probability measure on $\BB_\PP$.
The measure $\mu_{(\PP,m)}$ on $\Str_L$ is defined to be
$m^{\infty}\circ \F_{\PP}^{-1}$,
where $m^\infty$ is the product measure on $\BB_\PP^\omega$.
\end{definition}

Note that $m^\infty$ is invariant under arbitrary reordering of the indices.
We will obtain an invariant measure on $\Str_L$ by taking the distribution of the \emph{random} structure with underlying set $\Naturals$ corresponding to an $m$-i.i.d.\ sequence of elements of 
$\PP$.

Note that 
$\mu_{(\PP,m)}$ is is a probability measure, namely the distribution of a random
element in $\Str_L$
induced via $\F_{\PP}$ by an $m$-i.i.d.\ sequence on $\BB_\PP$.
The invariance of $m^\infty$
under the action of $\sym$ on $\BB_\PP^\omega$ yields the invariance of $\mu_{(\PP,m)}$ under the logic action.
In particular, $\mu_{(\PP,m)}$ is a $\sym$-invariant measure.

\begin{lemma}[{\cite[Lemma~3.5]{AFP}}]
\label{mu-is-invariant}
Let $\PP$ be a Borel $L$-structure,
and let $m$ be a probability measure on $\BB_\PP$.
Then $\mu_{(\PP, m)}$ is $\sym$-invariant.
\end{lemma}

Recall that a measure is said to be \defn{continuous} (or
\emph{atomless}) if it assigns measure zero to every singleton. 
The following lemma describes several key properties of the invariant measures obtained by sampling using continuous measures.

\begin{lemma}
\label{continuouslemma}
Let $\PP$ be a Borel $L$-structure,
and let $m$ be a continuous probability measure on $\BB$.  Then 
$\mu_{(\PP, m)}$ 
is an ergodic structure that is
concentrated on the 
union
	of isomorphism classes 
of countably infinite substructures of $\PP$.
\end{lemma}
\begin{proof}
Let
$\AA = (a_i)_{i\in \w}$ be
an $m$-i.i.d.\ sequence of elements of $\BB$.
Note that the induced countable structure $\F_{\PP}(\AA)$ is now a \emph{random $L$-structure}, i.e., an $\Str_L$-valued random variable, whose distribution is $\mu_{(\PP,m)}$.
Because $m$ is continuous, and since for any $k\ne \ell$ the random variables $a_k$
and $a_{\ell}$ are independent,  the sequence $\AA$ has no repeated entries, almost surely.
Hence
$\F_{\PP}(\AA)$
is almost surely isomorphic to
a countably infinite (induced) substructure of $\PP$.

	The measure $\mu_{(\PP, m)}$ is ergodic, as shown in \cite[Proposition~2.24]{AFKP}.
\end{proof}

The fact that $\mu_{(\PP, m)}$ is concentrated on substructures of $\PP$ (up to isomorphism)
will be extended in \cref{Lemma:approximable-type-omission}, where we show how to make it 
almost surely satisfy a
given theory.

\subsection{Weighted homomorphism densities}
\label{hom-densities}
Weighted homomorphism densities will allow us 
in \cref{convergence-lemma}
to show that the weak limit of finite sampled substructures of a Borel $L$-structure achieves the desired invariant measure on countably infinite structures.
\begin{definition}
	Suppose $L' \subseteq L$, and let
	$\M$ be an $L'$-structure and $\N$ an $L$-structure.
	Define a map $f$ from the underlying set of $\M$ to the underlying set of $\N$ to be a \defn{homomorphism} if it is a homomorphism from $\M$ to $\N|_{L'}$ (the $L'$-reduct of $\N$), and a \defn{full homomorphism} if it also preserves all non-relations of $L'$.
	Write $\Hom(\M, \N)$ and $\Full(\M, \N)$ to denote these respective classes of maps.
\end{definition}

\begin{definition}
Suppose $L' \subseteq L$, with $L'$ finite, and let
	 $\M$ be a finite $L'$-structure and $\N$ an arbitrary $L$-structure.
Let $m$ be a probability measure on the underlying set of $\N$. 
Define 
the \defn{weighted full homomorphism density} $\tf(\M, (\N, m))$ to be the probability that assigning the elements of $\M$ to elements of $\N$ in an $m$-i.i.d.\ way yields a full homomorphism, i.e., a map in which relations and non-relations are preserved. Specifically,
\[
	\tf(\M, (\N, m)) \defas \int_{\Full(\M,\, \N)} \ \mathrm{d}m^{|\M|},
\]
which reduces to 
\[
	\tf(\M, (\N, m)) = \sum_{f\in \Full(\M,\, \N)} \ \prod_{a\in\M} m\bigl(f(a)\bigr)
\]
when $\N$ is countable.
\end{definition}

This notion extends the case where $\M$ is a simple graph and $\N$ is an edge-weighted graph, as in, e.g., 
\cite[\S5.2.1]{MR3012035}.
For more details on full homomorphism densities, see \cite[\S3.1]{AFNP}.

\begin{proposition}
	\label{fullhom-determines-muPm}
	Let $\N$ be a finite $L$-structure and $m$ a probability measure on the underlying set of $\N$.
	Then $\mu_{(\N, m)}$ is completely determined by the sequence of numbers
	\[
		\bigl \{
			\tf(\M, (\N, m))
		\bigr \}_{\M},
	\]
	where $\M$ ranges over finite $L'$-structures, where $L' \subseteq L$ and $L'$ is finite .
\end{proposition}
\begin{proof}
	For any finite $L' \subseteq L$	and any $L'$ structure $\M$ with underlying set $\{0, \ldots, |\M| - 1\}$, let $p_\M$ be the
	qf-type of $\{0, \ldots, |\M| - 1\}$
		in $\M$.
		For any finite $L' \subseteq L$ and quantifier-free $L'$-formula $\varphi$ with $|\M|$-many free variables, 
		let $\nu_{L'}(\llrr{\varphi(n_0, \ldots, n_{|\M|-1})})$ be the quantity  
	\[			\sum\{ \tf(\M, (\N, m)) \st \M \text{ is an $L'$-structure and } \models p_\M \to \varphi \}.
			\]

			Let $L'$ and $L''$ be finite languages such that $L' \subseteq L'' \subseteq L$. 
			Let $\psi$ be a quantifier-free $L'$-formula; in particular, $\psi$ is also an $L''$-formula. 
			Note that $\nu_{L''}(\llrr{\psi(n_0, \ldots, n_{|\M|-1})})$ is the probability that an $m$-i.i.d.\ sequence of elements of $\N$ of length $|\M|$ satisfies $\psi$.
			Hence $\nu_{L''}(\llrr{\psi(n_0, \ldots, n_{|\M|-1})})$ does not depend on the choice of $L''$ extending $L'$, and so we may define
			$\nu(\llrr{\psi(n_0, \ldots, n_{|\M|-1})})$ to be $\nu_{L'}(\llrr{\psi(n_0, \ldots, n_{|\M|-1})})$.

			Define $\Theta$ to be the collection of $\llrr{\varphi(n_0, \ldots, n_{r-1})}$
			where
			$n_0, \ldots, n_{r-1}, r \in \Nats$  and $\varphi$ is a quantifier-free $L$-formula with $r$-many free variables.
			Note that $\Theta$ is a ring and $\nu$ is a pre-measure on $\Theta$ that agrees with
	$\mu_{(\N, m)}$.
By the Carath\'eodory extension theorem, the unique extension of $\nu$ to a measure 
	must be $\mu_{(\N, m)}$.
\end{proof}
For similar results in the special case of graphs, see \cite[Chapter~11]{MR3012035}. For related results and discussion, see also \cite[\S1.1--1.2]{kruckman-thesis}.

\section{Layerings}
\label{sec:layerings}

Now that we have a method for producing
an 
ergodic structure
from a
Borel $L$-structure with an associated measure,
we need a way to construct Borel $L$-structures such that the resulting 
ergodic structures  almost surely
satisfy 
the desired theory.  Our general method is to build a Borel $L$-structure as the limit of a directed system of finite structures, as we describe in this section. 
This construction is similar in spirit to the inverse limit construction of \cite{AFNP}.

Throughout this section,
we restrict $L$ to be a countable \emph{relational} language. 

\subsection{Basic layerings}
\label{Subsection:layerings}

Our methods rely on a notion
called a \emph{layering}, which we introduce here.

\begin{definition}
Suppose $\PP_0, \PP_1$ are Borel $L$-structures on underlying Borel spaces $\BB_0, \BB_1$ respectively. We say that a map $f\colon \PP_0 \rightarrow \PP_1$ is a \defn{Borel homomorphism} when $f$ is a homomorphism of $L$-structures and is a Borel function. 
\end{definition}

We will be interested in a specific kind of Borel $L$-structure which arises as the limit of a sequence of Borel homomorphisms. 

By \emph{countable Borel $L$-structure}, we mean a countable (finite or infinite) $L$-structure with the discrete $\sigma$-algebra (i.e., that generated by all singletons).
We define the notion of an $L$-layering in terms of a directed system of countable Borel $L$-structures, but 
in the $L$-layerings that we build in \cref{sec:existence}, these Borel $L$-structures will always be finite.

\begin{definition}
	We define an \defn{$L$-layering}
	$(\N, m, i)$ to be a directed system indexed by $(\omega, <)$
	satisfying the following,
	for each $n \in \w$.

\begin{itemize}
\item 
	The structure
		$\N_n$ is a countable Borel $L$-structure, whose underlying Borel set we denote by $N_n$.
	\item
		The measure $m_n$ is a probability measure on $N_n$ whose only measure zero set is $\emptyset$.

\item For each $k \leq n$, the map $i_{n,k} \colon \N_{n} \rightarrow \N_k$ is a measure-preserving surjective homomorphism. In particular we have
		$m_n(i_{n,k}^{-1}(a)) = m_k(a)$
		for every $a \in N_k$.
\end{itemize}
For $n\in\w$, we refer to $(\N_n, m_n)$ as \defn{level $n$} of $(\N, m, i)$.
\end{definition}

We now describe the limit of a layering. 

\begin{definition}
	Suppose $\hat{\N}  = (\N, m, i)$
is an $L$-layering. 
We define the \defn{limit} of $\hat{\N}$ to be the pair $(\N_\w, m_\w)$ satisfying the following.
\begin{itemize}
\item The $L$-structure $\N_\w$ has underlying set 
	\[
		N_\w\defas 
		\Bigl\{\{a_k\}_{k \in \w} \st 
		(\forall n \in \w) \, 
		\Bigl(a_n \in \N_n
	\ \And \ 
		(\forall k \leq n)  \,
	\bigl(i_{n,k}(a_n) = a_k\bigr)\Bigl )\Bigr\}.
	\]

\item 
	For $n\in\w$, the map $i_{\omega,n} \colon \N_\omega \to \N_n$ sends the tuple $\{a_k\}_{k\in\omega}$ to $a_n$.
\item 
	The Borel structure on $\N_\w$ is that generated by 
	sets of the form 
	$i_{\omega, n}^{-1}(b)$, where $b\in N_n$ and $n\in\w$.

\item For any relation symbol $R \in L$,  
	\[\N_\w \models \neg R(\{a^1_k\}_{k \in \w}, \dots, \{a^j_k\}_{k \in \w})\]
	if and only if 
	\[\N_n \models \neg R(a^1_n, \dots, a^j_n)\] for some $n\in \omega$, where $j$ is the arity of $R$. 

\item The measure $m_\w$ is the unique probability measure satisfying
	\[m_\w\bigl(\bigl\{\{a_k\}_{k \in \w} \st  a_j = b\bigr\}\bigr) = m_j(b)\]
	for all $j \in \w$ and $b \in \N_j$. 
\end{itemize}
\end{definition}

It is easy to check that this is well-defined; in particular,
	by the Carath\'eodory extension theorem
	the measure is uniquely determined by specifying the measures of the sets of the form
	$i_{\omega, n}^{-1}(b)$, where $b\in N_n$ and $n\in\w$.

	Observe that $\N_\w$ is the limit in the category of Borel $L$-structures and Borel homomorphisms of the directed system $\hat{\N}$.
	The key property of
$(\N_\w, m_\w)$ 
is that the ergodic structure $\mu_{(N_\w, m_\w)}$ obtained by sampling from $(\N_\w, m_\w)$ as in \cref{mu-is-invariant}
is the weak limit of the sequence of ergodic structures obtained by sampling from each of the $(\N_n, m_n)$ for $n\in\omega$. 
We now make this precise.  Let $\weakconv$ denote convergence in the weak (a.k.a.\ weak-$*$ or vague) topology.

\begin{lemma}
	\label{convergence-lemma}
	The sequence of measures $\{\mu_{(\N_n, m_n)}\}_{n\in\w}$ converges weakly to $\mu_{(\N_\w, m_\w)}$, i.e.,
	we have
	$\mu_{(\N_n, m_n)} \weakconv \mu_{(\N_\w, m_\w)}$.
\end{lemma}
\begin{proof}
	Fix a finite sublanguage $L' \subseteq L$ and a finite $L'$-structure $\M$. 
	For each $k\in \omega$, consider the weighted full homomorphism density
	$\tf(\M, (\N_k, m_k))$ of $\M$ in $(\N_k, m_k)$.
		Observe that because the map $i_{k+1,k} \colon \N_{k+1} \to \N_k$ is a homomorphism, the probability that the random map 
		$\M \to \N_k$
		(that independently assigns each element of $\M$ an element of $\N_k$ according to $m_k$)
is a homomorphism is at least the probability that the corresponding random map
$\M \to \N_{k+1}$ is a homomorphism.
	Therefore the sequence of positive reals $\{\tf(\M, (\N_k, m_k))\}_{k\in \omega}$ is non-increasing. 
	In particular, 
	the limit $\lim_{k\to \infty} \tf(\M, (\N_k, m_k))$ exists and is equal to $\tf(\M, (\N_\omega, m_\w))$, because 
	$(\N_\w, m_\w)$ is the limit of the directed system formed by $\{(\N_k, m_k)\}_{k\in\Nats}$ and the maps $\{i_{{k+1},k}\}_{k\in\Nats}$ in the category of probability spaces.
	Hence for any finite $L' \subseteq L$ and any quantifier-free $L'$-formula $\varphi$,
	and any $n_0, \ldots, n_{r-1} \in \Nats$, where $\varphi$ has $r$-many free variables,
	we have
	\[
		\lim_{k\to \infty} \mu_{(\N_k, m_k)}(\llrr{\varphi(n_0, \ldots, n_{r-1})})
		=
		\mu_{(\N_\w, m_\w)}(\llrr{\varphi(n_0, \ldots, n_{r-1})})
	\]
	by 
	\cref{fullhom-determines-muPm}.
	Note that every clopen set is of the form $\llrr{\varphi(n_0, \ldots, n_{r-1})}$.
	Therefore by the Portmanteau theorem, we have
	$\mu_{(\N_n, m_n)} \weakconv \mu_{(\N_\w, m_\w)}$.
\end{proof}

\subsection{Continuous limits}

We will mainly be interested in layerings whose limit measure is a continuous measure on the limit structure, i.e., where each singleton of the limit structure is a nullset of the limit measure. 
In this case we say that the layering is \defn{continuous}.
We now characterize such layerings.

\begin{lemma}
\label{Lemma: continuous layering}
Suppose $\hat{\N} = (\N, m, i)$ is an $L$-layering with limit $(\N_\w, m_\w)$. Then the following are equivalent.
\begin{enumerate}
	\item[(a)] The probability measure $m_\w$ is continuous. 

	\item[(b)] For all $k \in \w$ and $a \in \N_k$, there are $n > k$ and $X, Y \subseteq \N_n$ such that 
\begin{itemize} 
\item $X, Y\subseteq i_{n, k}^{-1}(a)$, 

\item $X \cap Y = \emptyset$,

\item $m_n(X) > 1/3 \cdot m_k(a)$, and
\item $m_n(Y) > 1/3 \cdot m_k(a)$. 
\end{itemize}
\end{enumerate}
\end{lemma}
\begin{proof}
	Suppose $m_\w$ is continuous. Towards a contradiction, assume that $k \in \w$ and $a \in \N_k$ are such that for all $n > k$ there are no $X, Y \subseteq \N_n$ satisfying the conditions of (b). 
	
	If for some $n>k$ every element $b \in \N_n$ satisfied $m_n(b) < 2/3 \cdot m_n(a)$, then we could construct $X$ and $Y$ satisfying the conditions of (b) by adding elements in decreasing mass to the set $X$ until its mass exceeded $1/3 \cdot m_k(a)$ and allowing $Y = N_n \setminus X$.

	Hence for every $n > k$ there is some $b_n \in \N_n$ such that $i_{n,k}(b_n) = a$ and $m_n(b_n) > 2/3 \cdot m_k(a)$;
	note that is exactly one such $b_n$.
	Therefore 
	we have
	$i_{n', n}(b_{n}) = b_{n'}$
	for $n'>n>k$.
	In particular, if we let $b_j = i_{k,j}(a)$ for $j \le k$, then $m_\w(\{b_j\}_{j\in \w}) > 2/3\cdot m_k(a)$, 
	contradicting the continuity of $m_\w$.

	Now suppose (b) holds, and let  $\{a_j\}_{j\in \w} \in \N_\w$. Then for every $k\in\w$ there is some $n > k$ such that 
	$m_n(a_n)  \le 2/3 \cdot m_k(a_k)$. Hence 
$m_\w (\{a_j\}_{j\in \w}) = \inf_j m_j(a_j),$  which is at most 
		$\left(\frac23\right)^{-q} \cdot m_0(a_0)$ 
		for all $q\in \w$, and so $m_\w (\{a_j\}_{j\in \w}) = 0$.
\end{proof}

\subsection{Regular layerings}

We will need two other specific conditions on a layering to allow us sufficient control over
the limit structure and measure of the layering.

First, we want to be able to define our layerings without specifying where all the mass is assigned at any given level;
this will allow us at later levels to \emph{add new elements}. We will accomplish this via a distinguished ``virtual'' element whose mass at each level consists of that mass not assigned to other ``real'' elements.

Second, at each level
we want to declare that certain first-order  quantifier-free formulas hold, and we want such formulas to continue to hold at later levels. Unfortunately, in arbitrary layerings different formulas may hold at different levels, as we 
required only that the maps between levels were homomorphisms, and not necessarily embeddings. 
We will address this by 
assigning each level a language such that the maps between a level and all higher ones preserve the qf-type (in the language of the level) of every
tuple that is non-redundant and whose image is also non-redundant.
In other words, for such tuples, formulas remain \emph{set}, once set at a lower level. 
Because we want to allow ourselves to split elements as we move up the levels of the layering, we do not require the qf-type of redundant tuples (or those whose image is redundant) to be set.

We now make both of these conditions precise. 

\begin{definition}
	A \defn{regular $L$-layering} is a layering $\hat{\N} = (\N, m, i)$ along with the following additional data.
\begin{itemize}
	\item For each $n\in\omega$ there is an element $*_n \in \N_n$, called the \defn{sink}
		of $\N_n$, such that 
	\begin{itemize}
		\item[(a)] all relations in $L$ hold of any tuple in $\N_n$ containing $*_n$,
	
		\item[(b)] for all $k \leq n$ we have $i_{n,k}(*_n) = *_k$, and
	
		\item[(c)] for all $\varepsilon > 0$ there is an $n \in \w$ such that $m_n(*_n) < \varepsilon$. 
	\end{itemize}

\item There is a sequence $\{L_n\}_{n \in \w}$ of countable relational languages such that for all $n\in\w$,
	\begin{itemize}
		\item[(d)] $L_n \subseteq L_{n+1}$,
	
		\item[(e)] $\bigcup_{j \in \w}L_j = L$,
	
		\item[(f)] for all $k \leq n$ and all non-redundant tuples $\a\in\N_n$, if
			$i_{n, k}(\a)$ is a non-redundant tuple in $\N_k$, then the  qf $L_k$-type of $\a$ is the same as that of $i_{n, k}(\a)$, and
	
		\item[(g)] 
			all tuples $\a\in \N_n$ and all $R \in L$ of arity $|\a|$, 
			if $\N_n \models \neg R(\a)$ then $R \in L_n$ and $\a$ is non-redundant.
	\end{itemize}
\end{itemize}
\end{definition}

We can think of a level as containing information about tuples in the resulting limit structure. 
Condition (a) asserts that the sink
$*_n$ is a \emph{neutral element} in that it doesn't force any structure at higher levels.
Condition (b) is straightforward, while (c) ensures that 
$m_\w(*_\omega) = 0$, where $*_\omega \defas \{*_j\}_{j\in\w}$.
Conditions (d) and (e) say that $L$ is an increasing union of sublanguages $L_n$, and (f) ensures that if a decision is made about a non-redundant tuple, that decision propagates to higher levels.
Condition (g) states that $\N_n$ is \emph{neutral with respect to $L \setminus L_n$} and makes no decisions about redundant tuples.

\subsection{Approximations to types}
We now introduce a condition on a continuous regular layering $(\N, m, i)$ which will ensure that 
$\mu_{(\N_\w, m_\omega)}$
omits a certain kind of type, which can be appropriately approximated in sublanguages.
We will also see that satisfying a pithy $\Pi_2$ sentence and omitting a qf-type can be expressed as omitting a type of this kind.

\begin{definition}
	Let $L$ be a countable relational language and let $\{L_j\}_{j\in\w}$ be a nondecreasing sequence of languages whose union is $L$.
Let $\hat{p}$ be a sequence $\{p_j\}_{j\in\w}$ of types such that each $p_j$ is a partial $L_j$-type in the same variables, and such that $\models p_k \to p_j$ for $j<k$.
Then $\hat{p}$ is an \defn{approximable $\{L_j\}_{j\in\w}$-type}  if
		for all $j < k$, for all $L_k$-structures $\M_k$ and $L_j$-structures $\M_j$, and for all 
		surjective embeddings
		$i_{k, j}$  from $\M_k$ to $\M_j$, if
		\begin{itemize}
			\item			$\a\in\M_k$ is a tuple satisfying $p_k$, and 
			\item 			$i_{k, j}$ is injective on the underlying set of $\a$,
		\end{itemize} 
		then the tuple $i_{k, j}(\a)$ satisfies $p_j$.
\end{definition}
Elements of the limit of a regular layering correspond to sequences with one element from each level of the layering, where the elements of the higher levels of the sequence project down to the elements at the lower levels of the sequence, via the maps of the layering. 
As such, a tuple of elements in the limit corresponds to a sequence of tuples from each level, on which the maps of the layering are injective past some point.

The idea behind the notion of an approximable type is to capture information about elements in the limit of a regular layering by considering properties of the corresponding sequence. 
However, because the approximable type is expressed in a way that is independent of the layering, we need it to cohere with any possible layering it could be asked to describe. This is the motivation behind the coherence condition with respect to arbitrary surjective embeddings.

\begin{definition}
	Let $\hat{\N} = (\N, m, i)$
	be a continuous regular $L$-layering, and suppose that $\{p_j\}_{j\in\w}$ is an approximable $\{L_j\}_{j\in\w}$-type.
	We say that  $\hat{\N}$
	\defn{asymptotically omits}  $\{p_j\}_{j\in\w}$ if 
	for all $k \in \w$ 
	there is an $n_k > k$ such that 
	\[
		\N_{n_k} \models \bigl(\forall \xx \in i_{n_k, k}^{-1}(\N_k \setminus *_k)\bigr)
		\,
		\neg p_k(\xx).
	\]
\end{definition}

The following key technical lemma will allow us to ensure that a limit measure satisfies or omits formulas of several particular forms.

\begin{lemma}
\label{Lemma:approximable-type-omission}
	Let $\hat{\N} = (\N, m, i)$
	be a continuous regular $L$-layering, and suppose $\{p_j\}_{j\in\w}$ is an approximable $\{L_j\}_{j\in\w}$-type.
	If $\hat{\N}$ asymptotically omits $\{p_j\}_{j\in\w}$ then 
	\[
		\mu_{(\N_\w, m_\w)}\Bigl( \Bigllrr{(\forall \xx) \bigvee_{j\in\w} \neg p_j(\xx)}\Bigr)  = 1.
	\]
\end{lemma}
\begin{proof}
	Let $\x$ be the variables of $\{p_j\}_{j\in\w}$
	and let $k = |\xx|$.
	For $j\in\w$, define $\alpha_j \defas m_j(*_j) + (1 - \beta_j)$,  where $\beta_j$ denotes the probability that sampling $k$ elements independently from $m_j$ yields a non-redundant tuple.
Note that $\alpha_j$
is an upper bound on the probability that 
the sink $*_j$ is selected at least once or that some element is selected at least twice.
Because $\hat{\N}$ is regular, $\lim_j m_j(*_j) = 0$. Because $\hat{\N}$ is continuous, $\lim_j \beta_j = 1$. Hence $\lim_j \alpha_j = 0$.

Note that
if $\ell \ge j$ and $\a$ is a non-redundant tuple of $\N_\ell\setminus *_\ell$ that
satisfies $p_j$, then 
	if $i_{\ell, j}(\a)$
	is also non-redundant it satisfies $p_j$.
	Hence, by the asymptotic omission of $\{p_j\}_{j\in\w}$, if $g \ge n_j$ or $g = \omega$, then
	the probability that a sample from the $k$-fold product measure $m_g^k$ 
	satisfies $p_j$
	is at most
$\alpha_j$. 

Therefore the probability that a sample from the $k$-fold product measure $m_\w^k$ satisfies
$(\exists \xx) \bigwedge_{j\in\w} p_j(\xx)$
is at most $\inf_j \alpha_j = 0$. 
Hence a sample from $m_\w^k$ almost surely satisfies 
$(\forall \xx) \bigvee_{j\in\w} \neg p_j(\xx)$, and so
\[
	\mu_{(\N_\w, m_\w)}\Bigl( \Bigllrr{(\forall \xx) \bigvee_{j\in\w} \neg p_j(\xx)}\Bigr)  = 1,
\]
as desired.
\end{proof}

We now consider two special classes of approximable types that can be asymptotically omitted.
Asymptotically omitting an approximable type of the first class in a 
continuous regular layering $(\N, m, i)$
ensures that the limit measure $\mu_{(\N_\w, m_\omega)}$ satisfies a given 
pithy $\Pi_2$-theory. 

\begin{definition}
	Let $\hat{\N} = (\N, m, i)$
	be a continuous regular $L$-layering. Suppose that
	$(\forall \xx)(\exists y) \varphi(\xx, y) \in \Lwow(L)$ is a pithy $\Pi_2$ sentence such that there exists some $\ell\in\w$ for which $(\forall \xx)(\exists y) \varphi(\xx, y) \in \Lwow(L_\ell)$.
	Consider the approximable type $\{p_j\}_{j\in\w}$ where 
	$p_j(\xx) = (\xx=\xx)$ for $j< \ell$ and $p_j(\xx) = (\forall y) \neg \varphi(\xx, y)$ for $j \ge \ell$.
	We say that $\hat{\N}$ 
	\defn{asymptotically satisfies}
	$(\forall \xx)(\exists y) \varphi(\xx,y)$ 
	when $\hat{\N}$ asymptotically omits $\{p_j\}_{j\in\w}$.	 (Note that this does not depend on the choice of $\ell$.)

	We say that $\hat{\N}$ asymptotically satisfies a pithy $\Pi_2$-theory $T$ of $\Lwow(L)$ when 
	it asymptotically satisfies every sentence of $T$.
\end{definition}

The following corollary is immediate from \cref{Lemma:approximable-type-omission}.

\begin{corollary}
\label{Cor:asymptotically satisfying a theory}
	Let $\hat{\N} = (\N, m, i)$
	be a continuous regular $L$-layering, and suppose 
	$T$ is a pithy $\Pi_2$-theory of $\Lwow(L)$. If $\hat{\N}$ asymptotically satisfies $T$ then $\mu_{(\N_\w, m_\w)}(\llrr{T}) = 1$, i.e., $\mu_{(\N_\w, m_\w)}$ is concentrated on $T$.
\end{corollary}

Asymptotically omitting an approximable type of the second class results in 
a limit measure $\mu_{(\N_\w, m_\omega)}$ that omits a given qf $L$-type.

\begin{definition}
	Let $\hat{\N} = (\N, m, i)$ be
	a continuous regular $L$-layering, and let $p$ be a qf $L$-type.
	We say that $\hat{\N}$ \defn{asymptotically omits} $p$ 
	when $\hat{\N}$ asymptotically omits  the approximable type
	$\Bigl\{ \bigwedge  \bigl(p  \cap  \Lww(L_j)\bigr)\Bigr\}_{j\in\w}$.
\end{definition}

Again the following corollary is immediate from \cref{Lemma:approximable-type-omission}.

\begin{corollary}
\label{Cor:asymptotically omitting types}
	Let $\hat{\N} = (\N, m, i)$
	be a continuous regular $L$-layering, and let $p(\xx)$ be a qf $L$-type. 
	If $\hat{\N}$ asymptotically omits $p$ then $\mu_{(\N_\w, m_\w)}(\llrr{ (\forall \xx) \neg p(\xx)}) = 1$,
	i.e., $\mu_{(\N_\w, m_\w)}$ omits the singleton set of types $\{p\}$.
\end{corollary}

\subsection{Layer Transformation}
\label{transformation}

We have reduced the problem of constructing an ergodic structure almost surely satisfying a theory to that of constructing a regular continuous layering which asymptotically omits a countable collection of types and asymptotically satisfies a pithy $\Pi_2$-theory. 
Now we give a general method for constructing  such layerings.

We do this by isolating a relationship between successive levels of a layering, such that when it occurs infinitely often along pairs of consecutive levels of the layering, 
the limit measure has the desired properties.

An \emph{approximation} consists of the data needed to be a level of a regular $L$-layering. 

\begin{definition}
	Let $\M$ be a non-redundant $L$-structure.
An \defn{approximation} to $\M$ is a triple $(L_+, \N_+, m_+)$ such that
\begin{itemize}
\item $L_+ \subseteq L$;

\item the underlying set $N_+$ of $\N_+$ satisfies $N_+ = M_+ \cup *_+$, where $M_+ \subseteq M$, the underlying set of $\M$;

\item 
	$\M\models R(\a)$ if and only if $\N_+ \models R(\a)$
	for any tuple $\a\in M_+$ and any relation symbol $R\in L_+$ of arity $|\a|$;

\item all relations in $L$ hold of any tuple in $\N_+$ containing $*_n$; 

\item $\N_+ \models R(\a)$
	for all tuples $\a\in \N_+$ and all $R \in L \setminus L_+$ of arity $|\a|$; and
\item $m_+$ is a measure on $N_+$.
\end{itemize} 
\end{definition}

A layer transformation is something which takes an approximation to $\M$, i.e., something which could be a level of a regular layering, and returns a new approximation along with a map to the old one. 

\begin{definition}
	Let $\M$ be a non-redundant $L$-structure.
A \defn{layer transformation} is a function $f$ which takes an 
approximation
$(L_+, \N_+, m_+)$ 
of $\M$
to a pair $(\iota_\dagger, (L_\dagger, \N_\dagger, m_\dagger))$ such that 
\begin{itemize}
\item $(L_\dagger, \N_\dagger, m_\dagger)$ is an approximation to $\M$;
\item $\N_+$ has sink $*_+$ and $\N_\dagger$ has sink $*_\dagger$;

\item $L_+ \subseteq L_\dagger$;

\item $\iota_\dagger\colon \N_\dagger \to \N_+$ is a surjective homomorphism and $i_{\dagger}(*_\dagger) = *_+$;

\item for all $\a\in\N_\dagger$, if
			$i_{\dagger}(\a)$ is a non-redundant tuple in $\N_+$, then the qf $L_+$-type of $\a$ is the same as that of $i_{\dagger}(\a)$; and

\item for all $a \in \N_+$, we have $m_\dagger(i_{\dagger}^{-1}(\{a\})) = m_+(\{a\})$. 
\end{itemize}
\end{definition}

We will show that properties of layer transformation transfer to properties of limit measures of continuous regular layerings which have the layer transformations as consecutive pairs of layers.

\begin{definition}
We say that a layer transformation \defn{forces} a property $P$ if whenever $(\N, m, i)$ is a $L$-layering such that for infinitely many $n \in \Nats$, the pair $(i_{n+1, n}, (L_{n+1}, \N_{n+1}, m_{n+1}))$ is the image of $(L_n, \N_n, m_n)$ under the transformation, then $(\N, m, i)$ has property $P$.

We say that a layer transformation \defn{continuously forces} a property $P$ if whenever $(\N, m, i)$ is a 
continuous regular 
$L$-layering such that for infinitely many $n \in \Nats$, the pair $(i_{n+1, n}, (L_{n+1}, \N_{n+1}, m_{n+1}))$ is the image of $(L_n, \N_n, m_n)$ under the transformation, then $(\N, m, i)$ has property $P$.
\end{definition}

Let $\M$ be a non-redundant $L$-structure.
We now introduce properties of layer transformations in terms of $\M$ which, when interleaved in a layering, will force the layering to be continuous, regular, and concentrated on a specified theory.

We first consider two properties which force a layering to be continuous and regular.

The \emph{scrapwork transformation} simply adds a new element to the structure and gives it half the mass of the sink of the previous layer. 

\begin{definition}
We define a \defn{scrapwork transformation}
to be any layer transformation which takes an approximation $(L_+, \N_+, m_+)$ and returns 
a pair $(\iota_\dagger, (L_\dagger, \N_\dagger, m_\dagger))$ such that 
\begin{itemize}

\item $\iota_\dagger$ is the identity on $N_+\setminus *_+$ and takes the element $a$ to $*_\dagger$; 
\item $L_\dagger = L_+$;
\item $N_\dagger = N_+ \cup \{a\}$ for some $a \in \M$; and

\item $m_\dagger(\{b\}) = m_+(\{b\})$ for any $b \in N_+\setminus *_+$ and $m_\dagger(a) = m_\dagger(*_\dagger) = \frac{1}{2} m_+(*_+)$.
\end{itemize}
\end{definition}

The following is then immediate. 

\begin{lemma}
\label{scrapwork transformation forces regular layering}
Any scrapwork transformation forces $m(*_\omega) =0$ and hence forces a layering to be regular. 
\end{lemma}
\begin{proof}
	First note that 
	$m_\dagger(*_\dagger) = \frac12 m_+(*_+)$, and so each application of the scrapwork transformation reduces then measure of the sink by one half. In particular, if there are infinitely many such applications, then $m_\omega(*_\omega) = 0$.
\end{proof}

A \emph{splitting transformation} splits each element into two new elements that each stand in the same relationship to all other elements as the original element did.

\begin{definition}
Suppose that for every quantifier-free formula $\psi$,
	the $\Lwow(L)$-theory of $\M$ has trivial $\psi$-definable closure.
Fix an enumeration $b_1, \ldots, b_k$ of $N_+ \setminus *_+$ and let $q$ be the quantifier free $L_+$-type of $b_1, \ldots, b_k$. By \cref{multiple-duplication} there are tuples $a_1^0\cdots a_k^0$  and $a_1^1\cdots a_k^1$ containing $2k$-many distinct elements of $\M$ such that $\M \models q(a_1^{\alpha(1)}, \dots, a_k^{\alpha(k)})$ for any function $\alpha \colon \{1, \dots, k\} \rightarrow \{0, 1\}$. 

We define a \defn{splitting transformation}
to be any layer transformation which takes an approximation $(L_+, \N_+, m_+)$ and returns 
a pair $(\iota_\dagger, (L_\dagger, \N_\dagger, m_\dagger))$  such that
\begin{itemize}
	\item $i_{\dagger}(a^j_\ell) = b_\ell$;
	\item $L_\dagger = L_+$;
	\item 
$N_\dagger = \{ a^j_\ell \st j \in \{0, 1\} \text{~and~} 1\le \ell \le k\} \cup *_\dagger$; and
	\item $m_{\dagger}(a^j_\ell) = m_{+}(b_\ell)/2$ for $j\in\{0, 1\}$ and $1 \le \ell \le k$ and let $m_\dagger(*_\dagger) = m_+(*_+)$
\end{itemize}
for some choice of elements $a^i_j$ and $b_j$ as above.
\end{definition}

\begin{lemma}
\label{splitting transformation implies continuous}
Any splitting transformation forces the measure $m_\omega$ to be continuous outside of the sink $*_\omega$. 
\end{lemma}
\begin{proof}
	Let $X_n = \sup\{m_n(\{a\}) \st a \in \N_n \setminus *_n\}$. 
If the relationship between levels $n$ and $n+1$ is a splitting transformation, then $X_{n+1}  = \frac{1}{2} X_n$. 
Therefore $\lim_{n \to \infty} X_n = 0$,
and so the splitting transformation forces $m_\omega$ to be continuous on $\N_\omega \setminus *_\omega$.
\end{proof}

Sampling from any structure with respect to a measure having a point mass will yield a random structure that almost surely has an indiscernible set. As such, if we hope to find a measure concentrated on structures or theories without an indiscernible set, then we need to be able to sample continuous measures --- which is why we need our theory to have trivial $\psi$-definable closure for some $\psi$. 

In particular, if a layering has cofinally many splitting and scrapwork transformations, then it must be a continuous regular transformation.

We now describe a layering transformation that will ensure that the limit measure satisfies a given pithy $\Pi_2$ sentence.

\begin{definition}
	Let $(\forall \x)(\exists y)\varphi(\x, y)$ be a pithy $\Pi_2$ sentence, and suppose
$\M \models (\forall \x)(\exists y)\varphi(\x, y)$.
Define 
a \defn{$(\forall \x)(\exists y)\varphi(\x, y)$-satisfaction transformation}
to be any layer transformation which takes an approximation $(L_+, \N_+, m_+)$ and returns 
$(\iota_\dagger, (L_\dagger, \N_\dagger, m_\dagger))$ such that 
\begin{itemize}
\item $\iota_\dagger$ is the identity on $N_+\setminus *_+$ and takes $N_\dagger \setminus N_+$ to $*_\dagger$;  
\item $L_\dagger = L_+$; 

\item 
$\N_\dagger$ is such that
	for every  $\a \in \N_+\setminus *_+$ there is some $b \in \N_\dagger$ such that $\M \models \varphi(\a, b)$; and

\item $m_\dagger(b) = m_+(b)$ for any $b \in N_+\setminus *_+$, and $m_\dagger(a) > 0$ otherwise.
\end{itemize}
\end{definition}

Note that we can always find such an $\N_\dagger$ as $\M\models (\forall \x)(\exists y)\varphi(\x,y)$.

\begin{lemma}
\label{satisfaction transformation forces asymptotically satisfied}
Any $(\forall \x)(\exists y) \varphi(\x, y)$-satisfaction transformation continuously forces the layering to asymptotically satisfy $(\forall \x)(\exists y)\varphi(\x,y)$. 
\end{lemma}
\begin{proof}
	Suppose $(\N, m, i)$ is a continuous regular layering
	and that $\a \in \N_\w$. Let $n$ be such that $i_{\w, n}(a_i) \neq i_{\w, n}(a_j)$ for distinct $a_i, a_j \in \a$.
	Let $n'>n$ be a stage at which a $(\forall \x)(\exists y) \varphi(\x, y)$-satisfaction transformation occurs. There is therefore some $b \in \N_{n'} \setminus *_{n'}$ such that $\N_{n'}\models \varphi(i_{\w, n'}(a_0), \dots, i_{\w, n'}(a_{k-1}), b)$. 

Let $b_\w \in \N_\w$ be such that $i_{\w, n'}(b_\w) = b$. Then $\N_\w \models \varphi(\a, b_\w)$ and hence $\N_\w \models (\exists y)\varphi(\a,y)$. But as $\a$ was arbitrary, we have $\N_\w \models (\forall \x)(\exists y)\varphi(\x, y)$. 
\end{proof}

We now describe the final layering transformation, which will ensure that the limit measure almost surely omits a given type.

\begin{definition}
Let $(\forall \x)(\exists y)\varphi(\x, y)$ be a pithy $\Pi_2$ sentence, and
suppose $p(\x)$ is a qf-type omitted in $\M$.
Define a \defn{omitting-$p(\x)$ transformation}
to be any layer transformation which takes an approximation $(L_+, \N_+, m_+)$ and returns 
$(\iota_\dagger, (L_\dagger, \N_\dagger, m_\dagger))$ such that 
\begin{itemize}

	\item $\iota_\dagger = \id$; 

\item $L_\dagger$ is a language such that for all $\a \in \N_\dagger$ of the same arity as $p$, there is a qf-free $L_\dagger$-formula $\eta(\x)$ such that 
$\neg \eta(\x) \in p(\x)$ and
$\M \models \eta(\a)$;

\item $N_\dagger = N_+$

\item $m_\dagger = m_+$.
\end{itemize}
\end{definition}

Note that we can always find such a language, since $\M$ omits $p(\x)$.

\begin{lemma}
\label{omitting type transformation forces omitting types}
Any omitting-$p(\x)$ transformation continuously forces the layering to asymptotically omit $p(\x)$. 
\end{lemma}
\begin{proof}
	Suppose $(\N, m, i)$ is a continuous regular layering
	and that $\a \in \N_\w$. Let $n$ be such that $i_{\w, n}(a_i) \neq i_{\w, n}(a_j)$ for distinct $a_i, a_j \in \a$. Let $n'>n$ be a stage at which the omitting-$p(\x)$ transformation occurs. Hence there is some quantifier-free formula $\eta(\x)$ such that $\neg \eta(\x) \in p(\x)$ and $\N_{n'+1} \models \eta(i_{\w, n'}(a_0), \dots, i_{\w, n'}(a_{k-1}))$. But then $\N_\w \models \eta(\a)$, and so $\N_\w \models \neg p(\a)$. However, because $\a$ was arbitrary, this means that $\N_\w$ omits $p(\x)$. 
\end{proof}

\section{Existence of ergodic structures}
\label{sec:existence}
Having developed
appropriate layer transformations, we may now show that for a relational language $L$, if a complete $A$-theory $T$ has trivial $A$-definable closure, then there is an ergodic model of $T$. 

We begin by using duplication of quantifier-free formulas to construct an appropriate layering.

\begin{proposition}
	\label{layering-existence}
	Let $L$ be a countable relational language, and $\M\in \Str_L$ be a non-redundant $L$-structure.
	Further let $T$ be a countable 
	pithy $\Pi_2$ theory, and $\Theta$ be a countable set of qf-types such that $\M \models T$ and $\M$ omits $\Theta$. Suppose that $T$ has duplication of quantifier-free formulas.
	Then there is a continuous regular layering $(\N, m, i)$ such that 
	$\mu_{(\N_\w, m_\w)}$ is an ergodic model of $T$ that
	omits $\Theta$.
\end{proposition}

\begin{proof}
	Let $\{Q_j\}_{j \in \w}$ be an enumeration of the relation symbols of $L$, and let $\{p_j\}_{j \in \w}$ be an enumeration of $\Theta$ in which each qf-type is enumerated infinitely often. Let $\{(\forall \xx)(\exists y)\psi_j(\xx, y)\}_{j \in \w}$ be an enumeration of $T$ in which each sentence is enumerated infinitely often and such that for every $j\in\w$, each relation symbol of $\psi_j$ is among $\{Q_k \st k < j\}$.

We construct a continuous regular layering $(\N, m, i)$ in stages, building one level at every stage. 
	At each stage $n$, we will identify a finite $L$-substructure $\M_n$ of $\M$ (necessarily non-redundant). Let $M_n$ denote the underlying set of $\M_n$, and choose the sink $*_n$ to not be in the underlying set $M$ of $\M$. The structure $\N_n$ will be the 
unique $L$-structure with underlying set $M_n \cup *_n$ such that
\begin{itemize}
	\item the sink $*_n$ is a neutral element; 
	\item for any relation symbol $R$ in $L \setminus L_n$ 
		and any tuple $\a$ of length $k$, where $k$ is the arity of $R$,  we have $\N_n \models R(\a)$; and
	\item $\M_n|_{L_n}$ equals the substructure of $\N_n|_{L_n}$ having underlying set $M_n$. 
\end{itemize}

\noindent \ul{Stage $-1$:}\\
Define $L_{-1} \defas \emptyset$. 
Let $\M_{-1}$ be the unique structure on the empty set (which necessarily omits $\Theta$),
and let $m_0(*_{-1}) = 1$.  \\


\noindent \ul{Stage $4n$:} Scratch work. \\
Define $L_{4n} \defas L_{4n-1}$. 
Let $a \in M \setminus M_{4n-1}$ and define $\M_{4n}$ to be $\M$ restricted to $M_{4n} \cup \{a\}$.

Define $i_{4n, 4n-1}$ to be the identity on $M_{4n-1}$ and such that 
$i_{4n, 4n-1}(a) = *_{4n-1}$.
	Finally, let $m_{4n}$ be the probability measure that agrees with $m_{4n-1}$ on all subsets of $M_{4n-1}$ and assigns $m_{4n}(a) = m_{4n}(*_{4n}) = \frac{1}{2} m_{4n-1}(*_{4n-1})$. 
\\


\noindent \ul{Stage $4n+1$:} Duplication. \\
Define $L_{4n+1} \defas L_{4n}$. Let 
$k = |M_{4n}|$ and
fix an enumeration $b_1, \ldots, b_k$ of  $M_{4n}$.
Let $q$ be the quantifier-free $L_{4n+1}$-type of $b_1, \ldots, b_k$.
By hypothesis, the theory $T$ has duplication of quantifier-free formulas.
Hence by \cref{multiple-duplication} there are tuples
$a_1^0\cdots a_k^0$  and
$a_1^1\cdots a_k^1$ containing $2k$-many distinct elements of $M$
such that $\M \models q(a_1^{\alpha(1)}, \dots, a_k^{\alpha(k)})$ for any function $\alpha \colon \{1, \dots, k\} \rightarrow \{0, 1\}$. 
Define $\M_{4n+1}$ to be the $L_A$-substructure of $\M$ with underlying set
\[
	\bigl\{ a^j_\ell \st j \in \{0, 1\} \text{~and~} 1\le \ell \le k\bigr\}.
\]
Finally, let $i_{4n+1, 4n}(a^j_\ell) = b_\ell$ and $m_{4n+1}(a^j_\ell) = \frac12 m_{4n}(b_\ell)$ for $j\in\{0, 1\}$ and $1 \le \ell \le k$. 
\\

\noindent \ul{Stage $4n+2$:} Witnesses to $\Pi_2$-statements. \\
Define $L_{4n+2} \defas L_{4n+1}$. 
Let $B \subseteq M$ be a finite set be such that for every $j \leq n$,
\[
	\M \models (\forall \xx \in M_{4n+1})\,(\exists y \in B)\,\psi_j(\xx, y);
\]
such a $B$ exists, as $\M \models T$ and $(\forall \xx)(\exists y)\psi_j(\xx, y) \in T$.
Define $\M_{4n+2}$ to be $\M$ restricted to have underlying set $M_{4n+1} \cup B$.

Let $i_{4n+2, 4n+1}$ be the identity on $M_{4n+1}$ and such that 
$i_{4n+2, 4n+1}(b) = *_{4n+1}$
for all 
$b \in B \setminus M_{4n+1}$. 
Let $m_{4n+2}$ be the probability measure that agrees with $m_{4n+1}$ on all subsets of $M_{4n+1}$ and evenly assigns positive mass 
to 
$*_{4n+2}$ and to every element of $B \setminus M_{4n+1}$. 
\\

\noindent \ul{Stage $4n+3$:} Omitting qf-types. \\
Define $\M_{4n+3} \defas \M_{4n+2}$. Let $L_{4n+3} \supseteq L_{4n+2}$ be a finite language containing $\{Q_j \st j \le n\}$ and such that 
$\M_{4n+3}$
omits every qf-type in $\{p_j \cap \Lww(L_{4n+3}) \st  j \leq n\}$; 
this is possible
as each qf-type in $\Theta$ is quantifier-free and $\M$ omits every qf-type in $\Theta$.
Let $i_{4n+3, 4n+2} \defas \id$
and
$m_{4n+3} \defas m_{4n+2}$.
\\

Finally, as required, for all $n$, and $k < n$, define $i_{n,k} \defas i_{k+1, k} \circ i_{k+2, k+1}\circ \dots \circ i_{n, n-1}$.
This completes the construction of $(N, m, i)$. \\

For every $n$ the map 
\[
	(L_{4n-1}, N_{4n-1}, m_{4n-1})
	\mapsto
	\bigl(i_{4n, 4n-1}, (L_{4n}, N_{4n}, m_{4n})\bigr)
\]
is a scrapwork transformation, and so by \cref{scrapwork transformation forces regular layering}, $(\N, m, i)$ is a regular layering.

By \cref{mu-is-invariant}, $\mu_{(\N_\w, m_\w)}$ is an invariant probability measure, and by \cref{continuouslemma} it is ergodic.

For every $n$, the map
\[
	(L_{4n}, N_{4n}, m_{4n}) 
	\mapsto 
	(i_{4n+1, 4n}, (L_{4n+1}, N_{4n+1}, m_{4n+1}))
\]
is a splitting transformation, and so by \cref{splitting transformation implies continuous}, $(\N_\w, m_\w)$ is a continuous layering, since $m_\w(*_\w) = 0$.

For every $n$, the map 
\[
	(L_{4n+1}, N_{4n+1}, m_{4n+1})
	\mapsto
	\big(i_{4n+2, 4n+1}, (L_{4n+2}, N_{4n+2}, m_{4n+2}) \bigr)
\]
is a $(\forall \x)(\exists y)\varphi_n(\x, y)$-satisfaction transformation, and so by \cref{satisfaction transformation forces asymptotically satisfied}, $(\N, m, i)$ asymptotically satisfies each sentence in $T$. Hence
by \cref{Cor:asymptotically satisfying a theory},
$\mu_{(\N_\w, m_\w)}$ is concentrated on $T$.

For every $n$, the map 
\[
	(L_{4n+2}, N_{4n+2}, m_{4n+2})
	\mapsto
	\bigl(i_{4n+3, 4n+2}, (L_{4n+3}, N_{4n+3}, m_{4n+3})\bigr)
\]
is a omitting-$p_n(\x)$ transformation,
and so by \cref{omitting type transformation forces omitting types}, $(\N, m, i)$ asymptotically omits each qf-type in $\Theta$.
Hence by \cref{Cor:asymptotically omitting types},
$\mu_{(\N_\w, m_\w)}$ omits $\Theta$.
\end{proof}

Using this layering, we may now build an ergodic model of the theory.

\begin{proposition}
\label{Proposition: Existence of Invariant Measure}
Let $L$ be a countable language.
	Suppose $A$ is a countable fragment of $\Lwow(L)$ and $T$ is a complete $A$-theory.
	If $T$ has trivial $A$-definable closure
	then there is an ergodic model of $T$.
\end{proposition}
\begin{proof}
	Recall from \cref{pithy} the relational language $L_A$, the non-redundant pithy $\Pi_2$ first-order $L_A$-theory $\Th_A$, and the countable collection of qf-types $\Theta_A$ of $L_A$.  Further recall the pithy $\Pi_2$ non-redundant $L_A$-theory $T_A = \Th_A \cup \, \{R_\varphi \st \varphi \in T\}$. 

	Suppose $T$ has trivial $A$-definable closure.  Then $T_A$ has trivial $A^+$-definable closure as well by \cref{bijection-preserves-trivial-dcl}.  By \cref{bijection-preserves-invariant-measures} there is an ergodic model of $T$ if and only if there is an ergodic model of $T_A^+$.  Hence our problem is reduced to that of constructing an ergodic model of $T_A^+$, i.e., an ergodic $L_A$-structure almost surely satisfying $T_A$ and omitting $\Theta_A$. 

	The $A$-theory $T$ is countable and consistent; hence
	let $\cK$ be a model of $T$.  By \cref{theory-Morleyization-lift}, the $(L_A \cup L)$-structure $\cK_A$ is a model of $T_A$ that omits $\Theta_A$. Define the $L_A$-structure $\M\defas \cK_A|_{L_A}$, i.e., the pithy $\Pi_2$ Morleyization of $\cK$, which is also  a model of $T_A$ that omits $\Theta_A$.  Further, $\M$ is non-redundant (because $\Th_A$ is non-redundant by \cref{nonredundant-ThA}), as is every finite substructure of $\M$.  

	The theory $T_A$ has trivial $A^+$-definable closure, 
and so by \cref{trivial dcl implies duplications}
it also has duplication of formulas in $A^+$.
In particular, $T_A$ has duplication of quantifier-free formulas.
	
	By \cref{layering-existence}, there is a continuous regular layering $(\N, m, i)$ such that $\mu_{(\N_\w, m_\w)}$ is an ergodic model of $T_A$ that omits $\Theta_A$,
	as desired.
\end{proof}

\section{Non-existence of ergodic structures}
\label{nonexistence}
We now show that if a complete $A$-theory has non-trivial $A$-definable closure 
then it has no ergodic model.
As in \cref{sec:existence},
we work with
countable \emph{relational} languages $L$, but in the proof of \cref{main-result} in \cref{classification}, we will see that this holds for arbitrary countable languages.

\begin{proposition}
\label{condition for no invariant measure}
Let $L$ be a countable relational language, and let $A$ be a countable fragment of $\Lwow(L)$.
Suppose $T$ is a complete $A$-theory that has non-trivial $A$-definable closure. Then there does not exist an ergodic model of
	$T$.
\end{proposition}

\begin{proof}
	By Lemmas~\ref{trivial-dcl-formulas} and
	\ref{nonredundant-nontrivial-dcl}, there is some non-redundant formula $\varphi(\xx, y)$ 
	such that
	$
	T \models (\exists \xx) (\exists^{=1}y)\, \varphi(\xx, y).
	$

Suppose, towards a contradiction, that there is an 
ergodic model $\mu$ of $T$.
	Define $P \defas \bigllrr{(\exists^{=1} y)\varphi(0, 1, \dots, |\xx|-1, y)}$.
First note that if 
	$\mu(P )=0$,
	then by invariance
\[
	\mu\bigl(\bigllrr{(\exists^{=1}y)\varphi(n_0, n_1, \dots, n_{|\xx|-1}, y)}\bigr) = 0
\]
for any distinct elements  $n_0, n_1, \dots, n_{|\xx|-1}\in \w$.
	Hence $\mu\bigl(\bigllrr{(\exists\xx)(\exists ^{=1} y)\varphi(\xx, y)}\bigr) = 0$, contradicting the fact that $T \models (\exists \xx) (\exists^{=1}y)\, \varphi(\xx, y)$.
	Therefore  $\mu(P) > 0$.

Define the probability measure $\mu^\star$ to be the conditional distribution of $\mu$ given that $P$ holds, i.e.,
	\[\mu^\star(X) = \dfrac{\mu(P \cap X)}{\mu(P)}
\]
for all Borel sets $X$. 
	Note that $\mu^\star$ is invariant under any permutation of $\w$ that fixes $0, 1, \dots, |\xx|-1$.

By the definition of $P$, we have 
\[
	\llrr{\varphi(0, 1, \dots, |\xx|-1, i)} \cap
	\llrr{\varphi(0, 1, \dots, |\xx|-1, j)} \cap P = \emptyset
\]
for $i,j$ satisfying $|\xx| - 1 < i < j$. Because $\varphi$ is non-redundant, we also have 
\[
	\bigcup_{i \leq |\xx|-1} \bigllrr{\varphi(0, 1, \dots, |\xx|-1, i)} = \emptyset.
\]
Hence
\[
	1 = \mu^\star(P)
	= \sum_{i \ge |\xx|}\mu^\star\bigl(\bigllrr{\varphi(0, 1, \dots, |\xx|-1, i)}\bigr).
\]

But if $\alpha = \mu^\star\bigl(\bigllrr{\varphi(0, 1, \dots, |\xx|-1, |\xx|)}\bigr)$ then by invariance of $\mu^\star$ we also have $\alpha = \mu^\star\bigl(\bigllrr{\varphi(0, 1, \dots, |\xx|-1, i)}\bigr)$ for any $i \geq |\xx|$.
Hence $1 = \sum_{i \geq |\xx|} \alpha$, which is a contradiction.

Therefore there is no ergodic model of $T$.
\end{proof}

\section{Classification of ergodic structures}
\label{classification}
Putting
the results of Sections~\ref{sec:existence} and \ref{nonexistence} together,
we obtain our main theorem.

\proofof{\cref{main-result}}
	Let $L$ be a countable language (not necessarily relational), and let $\Sigma \subseteq \Lwow(L)$ be countable.

	First observe that (2) immediately implies (1). On the other hand, (1) implies (2) by \cref{ergodic-lemma}.
	Further, (4) implies (3) by letting $A$ be any countable fragment such that $\Sigma\subseteq A$ (for example, the fragment generated by $\Sigma$).
To conclude the proof, we will show that (3) implies (2) and that 
	(2) implies (4).

Let $A$ and $T$ be as in (3).
	Consider the countable relational language $L_A$, the countable fragment $A^+$,
	and the 
	$A^+$-theory $T_A^+$. Recall, by
	\cref{TAplus-complete-Aplus-theory}, that $T_A^+$ is a complete $A^+$-theory.
	By \cref{bijection-preserves-trivial-dcl},
	$T_A^+$ has trivial $A^+$-definable closure,
		as $T$ has trivial $A$-definable closure.
Hence by \cref{Proposition: Existence of Invariant Measure}, there is
	an ergodic model of $T_A^+$.
By \cref{bijection-preserves-invariant-measures},
	there is also an ergodic model of $T$, and hence of $\Sigma$,
	and so (2) holds.

	Now suppose that (2) holds, and let $\mu$ be an ergodic 
	model of $\Sigma$.
	Let $A$ be an arbitrary countable fragment such that $\Sigma\subseteq A$.
	By \cref{complete-consistent},
$\Th(\mu)$ is a complete 
$\Lwow(L)$-theory.
	Hence $\Th(\mu) \cap A$ is a complete $A$-theory 
	that admits an invariant measure,
	and is such that $\Sigma \subseteq \Th(\mu) \cap A$.
	By \cref{condition for no invariant measure}, the theory $\Th(\mu) \cap A$ must have trivial $A$-definable closure, and so (4) holds.
\Endproofof

We obtain \cref{main-result-FO} as a corollary by specializing to 
the fragment $\Lww(L)$ of first-order $L$-formulas.

\proofof{\cref{main-result-FO}}
Suppose $\Sigma$ admits an invariant measure; then (1) of \cref{main-result} holds. Hence taking $A = \Lww(L)$ in (4), there is a complete first-order theory $T \subseteq \Lww(L)$ with $\Sigma\subseteq T$ such that $T$ has trivial $\Lww(L)$-definable closure.

Conversely, suppose $T \subseteq\Lww(L)$ is a complete first-order theory such that
$\Sigma \subseteq T$  and $T$ has
has trivial $\Lww(L)$-definable closure. 
Then (3) of \cref{main-result} holds. Therefore by (1) there is an invariant measure concentrated on $T$ and hence on $\Sigma$.
\Endproofof

\section*{Acknowledgements}

The authors would like to thank Alex Kruckman for helpful discussions and detailed comments on a draft, and Haim Gaifman for illuminating conversations about the genesis of his work relating logic and probability. 

This research was facilitated by participation in 
the Trimester Program on Universality and Homogeneity of the Hausdorff Research Institute for Mathematics at the University of Bonn (September--December 2013), the 
LMS--EPSRC Durham Symposium on Permutation Groups and Transformation Semigroups at Durham University (July 2015), the workshop on 
Logic and Random Graphs at the Lorentz Center (August--September 2015), and the workshop on Homogeneous Structures at the Banff International Research Station for Mathematical Innovation and Discovery (November 2015).

Work on this publication by C.\,F.\ was made possible through the support of ARO grant W911NF-13-1-0212 and grants from the John Templeton Foundation and Google. The opinions expressed in this publication are those of the authors and do not necessarily reflect the views of the U.S.\ Government or the John Templeton Foundation.


\def\cprime{$'$} \def\polhk#1{\setbox0=\hbox{#1}{\ooalign{\hidewidth
  \lower1.5ex\hbox{`}\hidewidth\crcr\unhbox0}}}
  \def\polhk#1{\setbox0=\hbox{#1}{\ooalign{\hidewidth
  \lower1.5ex\hbox{`}\hidewidth\crcr\unhbox0}}} \def\cprime{$'$}
  \def\cprime{$'$} \def\cprime{$'$} \def\cprime{$'$} \def\cprime{$'$}
  \def\cprime{$'$} \def\cprime{$'$} \def\cprime{$'$}
\providecommand{\bysame}{\leavevmode\hbox to3em{\hrulefill}\thinspace}
\providecommand{\MR}{\relax\ifhmode\unskip\space\fi MR }
\providecommand{\MRhref}[2]{%
  \href{http://www.ams.org/mathscinet-getitem?mr=#1}{#2}
}
\providecommand{\href}[2]{#2}

\end{document}